\documentclass{amsart}

\usepackage{amsmath, amssymb}
\usepackage[all]{xy}

\theoremstyle{plain}
\newtheorem{theorem}{Theorem}[subsection]
\newtheorem{prop}[theorem]{Proposition}

\newtheorem{corollary}[theorem]{Corollary}
\newtheorem{claim}[theorem]{Claim}

\theoremstyle{remark}

\let\a\alpha
\let\b\beta

\let\z\zeta

\let\k\kappa
\let\l\lambda
\let\m\mu

\let\f\varphi
\let\L\Lambda

\let\Om\Omega

\def\lra{\longrightarrow}
\def\egal{\ar@{=}}
\def\surj{\ar@{->>}}

\def\A{\mathcal A}
\def\B{\mathcal B}
\def\C{\mathbb C}
\def\CC{\mathcal C}
\def\E{\mathcal E}
\def\F{\mathcal F}
\def\G{\mathcal G}
\def\HH{\mathcal H}
\def\I{\mathcal I}
\def\J{\mathcal J}
\def\K{\mathcal K}
\def\P{\mathbb P}
\def\S{\mathcal S}

\def\O{\mathcal O}
\def\TT{\mathcal T}
\def\U{\mathbb U}

\def\W{\mathbb W}

\def\Ker{{\mathcal Ker}}
\def\Coker{{\mathcal Coker}}
\def\Im{{\mathcal Im}}

\def\D{{\scriptscriptstyle \operatorname{D}}}
\def\T{{\scriptscriptstyle \operatorname{T}}}
\def\ss{{\scriptstyle \operatorname{ss}}}
\def\st{{\scriptstyle \operatorname{s}}}

\def\EE{\operatorname{E}}
\def\H{\operatorname{H}}
\def\h{\operatorname{h}}
\def\M{\operatorname{M}_{{\mathbb P}^2}}
\def\N{\operatorname{N}}
\def\pp{\operatorname{p}}
\def\PP{\operatorname{P}}
\def\SS{\operatorname{S}}

\def\Hom{\operatorname{Hom}}
\def\Aut{\operatorname{Aut}}
\def\Ext{\operatorname{Ext}}
\def\GL{\operatorname{GL}}
\def\length{\operatorname{length}}
\def\mult{\operatorname{mult}}
\def\rank{\operatorname{rank}}

\def\Hilb{\operatorname{Hilb}_{{\mathbb P}^2}}

\def\tensor{\otimes}
\def\isom{\simeq}

\newcommand{\noi}{\noindent}

\def\ds{\displaystyle}
\def\ba{\begin{array}}
\def\ea{\end{array}}

\begin{document}


\title[The classification of semi-stable plane sheaves supported on sextics]
{The classification of semi-stable plane sheaves supported on sextic curves}

\author[mario maican]{mario maican}

\address{Mario Maican \\
Institute of Mathematics of the Romanian Academy \\
Calea Grivi\c tei 21 \\
010702 Bucharest \\
Romania}

\email{mario.maican@imar.ro}

\begin{abstract}
We classify all Gieseker semi-stable sheaves on the complex projective plane
that have dimension 1 and multiplicity 6.
We decompose their moduli spaces into strata which occur naturally as quotients
modulo actions of certain algebraic groups. In most cases we give concrete
geometric descriptions of the strata.
\end{abstract}

\maketitle

\begin{center}
\end{center}


\section*{Acknowledgements.} \noindent The author was supported by the
Consiliul Na\c tional al Cercet\u arii \c Stiin\c tifice, grant PN II--RU 169/2010 PD--219.


\section{Introduction and summary of results}

Let $\M(r,\chi)$ denote the moduli space of Gieseker semi-stable sheaves on $\P^2(\C)$
with Hilbert polynomial $\PP(m)=rm+\chi$, $r$ and $\chi$ being fixed integers, $r \ge 1$.
Le Potier \cite{lepotier} found that $\M(r,\chi)$ is an irreducible projective variety of dimension $r^2+1$,
smooth at points given by stable sheaves and rational if $\chi \equiv 1$ or $2 \mod r$.
In \cite{drezet-maican} and \cite{mult_five} were classified all semi-stable
sheaves giving points in $\M(4,\chi)$ and $\M(5,\chi)$, for all values of $\chi$.
These moduli spaces were shown to have natural stratifications given by cohomological conditions
on the sheaves involved.
In this paper we apply the same methods to the study of sheaves giving points in the moduli spaces
$\M(6,\chi)$ and we succeed in finding a complete classification for such sheaves.
We refer to the introductory section of \cite{drezet-maican}
for a motivation of the problem and for a brief historical context.

In view of the obvious isomorphism $\M(r,\chi) \isom \M(r,\chi + r)$ and of the duality isomorphism
$\M(r,\chi) \isom \M(r,-\chi)$ of \cite{maican-duality}, it is enough, when $r=6$, to consider only the cases
when $\chi = 1, 2, 3, 0$. Each of these cases is dealt with in sections 3, 4, 5, respectively 6.
In section 2 we gather some general results for later use and, for the convenience of the reader,
we review the Beilinson monad and spectral sequences.
For a more detailed description of the techniques that we use the reader is referred to the
preliminaries section of \cite{drezet-maican}.
In the remaining part of this section we summarise our classification results.

\subsection{Notations and conventions} 

\begin{align*}
\M(r,\chi) =
& \text{ the moduli space of Gieseker semi-stable sheaves on $\P^2$} \\
& \text{ with Hilbert polynomial $\PP(m)=rm+\chi$;} \\
\N(n,p,q) =
& \text{ the Kronecker moduli space of semi-stable $q \times p$-matrices with} \\
& \text{ entries in a fixed $n$-dimensional vector space over $\C$, cf. 2.4 \cite{drezet-maican};} \\
\Hilb(n) =
& \text{ the Hilbert scheme of $n$ points in $\P^2$;} \\
\Hilb(d,n) =
& \text{ the flag Hilbert scheme of curves of degree $d$ in $\P^2$} \\
& \text{ containing $n$ points;} \\
V =
& \text{ a fixed vector space of dimension $3$ over $\C$;} \\
\P^2 =
& \text{ the projective plane of lines in $V$;} \\
\O (d) =
& \text{ the structure sheaf of $\P^2$ twisted by $d$;} \\
n \A =
& \text{ the direct sum of $n$ copies of the sheaf $\A$;} \\
\{ X, Y, Z \} =
& \text{ basis of $V^*$;} \\
\{ R, S, T \} =
& \text{ basis of $V^*$;} \\
[\F] =
& \text{ the stable-equivalence class of a sheaf $\F$;} \\
\F^\D =
& \text{ ${\mathcal Ext}^1(\F, \omega_{\P^2})$ if $\F$ is a one-dimensional sheaf on $\P^2$;} \\
X^\D =
& \text{ the image in $\M(r,-\chi)$ or in $\M(r,r-\chi)$, as may be the case,} \\
& \text{ of a set $X \subset \M(r,\chi)$ under the duality morphism;} \\
X^\st =
& \text{ the open subset of points given by stable sheaves inside a set $X$;} \\
\PP_{\F} =
& \text{ the Hilbert polynomial of a sheaf $\F$;} \\
\pp(\F) =
& \text{ $\chi/r$, the slope of a sheaf $\F$, where $\PP_{\F}(m)=rm+\chi$;} \\
\C_x, \C_y, \C_z =
& \text{ the structure sheaves of closed points $x, y, z \in \P^2$;} \\
\O_L =
& \text{ the structure sheaf of a line $L \subset \P^2$.}
\end{align*}
We say that a morphism $\f \colon p\O(m) \to q \O(n)$ is \emph{semi-stable as a Kronecker module}
if it is semi-stable in the sense of GIT for the canonical action by conjugation of
$(\GL(p,\C) \times \GL(q,\C))/\C^*$.
We represent $\f$ by a $q \times p$-matrix with entries in $\SS^{n-m}V^*$.
Semi-stability means that for every zero-submatrix of size $q' \times p'$ of any matrix representing $\f$
we have the relation
\[
\frac{p'}{p} + \frac{q'}{q} \le 1.
\]
We will often encounter the case when $q = p+1$. In this case $\f$ is semi-stable as a Kronecker module
precisely if it is not in the orbit of a morphism of the form
\[
\left[
\ba{cc}
\star & \psi \\
\star & 0
\ea
\right],
\]
where $\psi \colon r\O(m) \to r\O(n)$, $1 \le r \le p$.
A morphism $\f \colon 2\O(-1) \to 3\O$ is semi-stable precisely if the maximal minors of any matrix
representing $\f$ are linearly independent.
We refer to 2.4 \cite{drezet-maican} for a more general discussion about Kronecker modules
and their moduli spaces.

\subsection{The moduli space $\M(6,1)$} 

This moduli space can be decomposed into five strata: an open stratum $X_0$,
two locally closed irreducible strata $X_1, X_2$ of codimensions $2$, respectively $4$,
a locally closed stratum that is the disjoint union of two irreducible locally closed subsets
$X_3$ and $X_4$, each of codimension $6$, and a closed irreducible stratum $X_5$ of codimension $8$.
The stratum $X_0$ is an open subset inside a fibre bundle with fibre $\P^{17}$ and base $\N(3,5,4)$;
$X_2$ is an open subset inside a fibre bundle with fibre $\P^{21}$ and base $Y \times \P^2$,
where $Y$ is the smooth projective variety of dimension $10$ constructed at \ref{3.2.1};
$X_3$ is an open subset inside a fibre bundle with fibre $\P^{23}$ and base $\P^2 \times \N(3,2,3)$.
The closed stratum $X_5$ is isomorphic to $\Hilb(6,2)$.

Each locally closed subset $X_i \subset \M(6,1)$ is defined by the cohomological conditions
listed in the second column of Table 1 below.
We equip $X_i$ with the canonical induced reduced structure.
In the third column of Table 1 we describe, by means of locally free resolutions of length $1$,
all semi-stable sheaves $\F$ on $\P^2$ whose stable-equivalence class is in $X_i$.
Thus, for each $X_i$ there are sheaves $\A_i$, $\B_i$ on $\P^2$, that are direct sums of line bundles,
such that each sheaf $\F$ giving a point in $X_i$ is the cokernel of some morphism $\f \in \Hom(\A_i, \B_i)$.
The linear algebraic group $G_i = (\Aut(\A_i) \times \Aut(\B_i))/\C^*$ acts by conjugation on
the finite dimensional vector space $\W_i = \Hom(\A_i,\B_i)$.
Here $\C^*$ is identified with the subgroup of homotheties of $\Aut(\A_i) \times \Aut(\B_i)$.
Let $W_i \subset \W_i$ be the locally closed subset of injective morphisms $\f$ 
satisfying the conditions from the third column of the table.
We equip $W_i$ with the canonical induced reduced structure.
In each case we shall prove that the map $W_i \to X_i$ defined by $\f \mapsto [\Coker(\f)]$
is a geometric quotient map for the action of $G_i$.

\begin{table}[!hpt]{Table 1. Summary for $\M(6,1)$.}
\begin{center}
{\small
\begin{tabular}{|c|c|c|c|}
\hline \hline
{}
&
{\tiny cohomological conditions}
&
{\tiny classification of sheaves $\F$ giving points in $X_i$}
&
{\tiny codim.}
\\
\hline
$X_0$
&
\begin{tabular}{r}
$\h^0(\F(-1))=0$ \\
$\h^1(\F)=0$\\
$\h^0(\F \tensor \Om^1(1))=0$
\end{tabular}
&
\begin{tabular}{c}
{} \\
$0 \lra 5\O(-2) \stackrel{\f}{\lra} 4\O(-1) \oplus \O \lra \F \lra 0$ \\
$\f_{11}$ is semi-stable as a Kronecker module \\
{}
\end{tabular}
&
0
\\
\hline
$X_1$
&
\begin{tabular}{r}
{} \\
{} \\
$\h^0(\F(-1))=0$ \\
$\h^1(\F)=1$\\
$\h^0(\F \tensor \Om^1(1))=0$ \\
{} \\
{}
\end{tabular}
&
\begin{tabular}{c}
$0 \lra \O(-3) \oplus 2\O(-2) \stackrel{\f}{\lra} \O(-1) \oplus 2\O \lra \F \lra 0$ \\
$\f$ is not equivalent to a morphism of the form \\
${\ds
\left[
\ba{ccc}
\star & 0 & 0 \\
\star & \star & \star \\
\star & \star & \star
\ea
\right], \left[
\ba{ccc}
\star & \star & 0 \\
\star & \star & 0 \\
\star & \star & \star
\ea
\right], \left[
\ba{ccc}
\star & \star & \star \\
\star & \star & \star \\
\star & 0 & 0
\ea
\right], \left[
\ba{ccc}
0 & 0 & \star \\
\star & \star & \star \\
\star & \star & \star
\ea
\right]
}$
\end{tabular}
&
2
\\
\hline
$X_2$
&
\begin{tabular}{r}
$\h^0(\F(-1))=0$ \\
$\h^1(\F)=1$\\
$\h^0(\F \tensor \Om^1(1))=1$
\end{tabular}
&
\begin{tabular}{c}
{} \\
$0 \to \O(-3) \oplus 2\O(-2) \oplus \O(-1) \stackrel{\f}{\to} 2\O(-1) \oplus 2\O \to \F \to 0$ \\
${\ds \f = \left[
\ba{cccc}
q_1 & \ell_{11} & \ell_{12} & 0 \\
q_2 & \ell_{21} & \ell_{22} & 0 \\
f_1 & q_{11} & q_{12} & \ell_1 \\
f_2 & q_{21} & q_{22} & \ell_2
\ea
\right]}$ \\
$\ell_1, \ell_2$ are linearly independent, $d = \ell_{11} \ell_{22} - \ell_{12} \ell_{21} \neq 0$, \\
${\ds \left|
\ba{cc}
q_1 & \ell_{11} \\
q_2 & \ell_{21}
\ea
\right|}$, ${\ds \left|
\ba{cc}
q_1 & \ell_{12} \\
q_2 & \ell_{22}
\ea
\right| }$ are linearly independent modulo $d$ \\
{}
\end{tabular}
&
4
\\
\hline
$X_{3}$
&
\begin{tabular}{r}
$\h^0(\F(-1))=0$ \\
$\h^1(\F)=2$\\
$\h^0(\F \tensor \Om^1(1))=2$
\end{tabular}
&
\begin{tabular}{c}
{} \\
$0 \lra 2\O(-3) \oplus 2\O(-1) \stackrel{\f}{\lra} \O (-2) \oplus 3\O \lra \F \lra 0$ \\
$\f_{11}$ has linearly independent entries \\
$\f_{22}$ has linearly independent maximal minors \\
{}
\end{tabular}
&
6
\\
\hline
$X_{4}$
&
\begin{tabular}{r}
$\h^0(\F(-1))=1$ \\
$\h^1(\F)=2$\\
$\h^0(\F \tensor \Om^1(1))=3$
\end{tabular}
&
\begin{tabular}{c}
{} \\
$0 \to 2\O(-3) \oplus \O(-2) \stackrel{\f}{\to} \O (-2) \oplus \O(-1) \oplus \O(1) \to \F \to 0$ \\
${\ds \f= \left[
\ba{ccc}
0 & 0 & 1 \\
q_1 & q_2 & 0 \\
g_1 & g_2 & 0
\ea
\right], }$ \\
where $q_1, q_2$ have no common factor or \\
${\ds \f= \left[
\ba{ccc}
\ell_1 & \ell_2 & 0 \\
q_1 & q_2 & \ell \\
g_1 & g_2 & h
\ea
\right] }$, \\
where $\ell_1, \ell_2$ are linearly independent, $\ell \neq 0$ \\
and $\f$ is not equivalent to a morphism of the form \\
${\ds \left[
\ba{ccc}
\star & \star & 0 \\
0 & 0 & \star \\
\star & \star & \star
\ea
\right]}$ \\
{}
\end{tabular}
&
6
\\
\hline
$X_{5}$
&
\begin{tabular}{r}
$\h^0(\F(-1))=1$ \\
$\h^1(\F)=3$\\
$\h^0(\F \tensor \Om^1(1))=4$
\end{tabular}
&
\begin{tabular}{c}
{} \\
$0 \lra \O(-4) \oplus \O(-1) \stackrel{\f}{\lra} \O \oplus \O(1) \lra \F \lra 0$ \\
$\f_{12} \neq 0$, $\f_{12} \nmid \f_{22}$ \\
{}
\end{tabular}
&
8
\\
\hline \hline
\end{tabular}
}
\end{center}
\end{table}

\subsection{The moduli space $\M(6,2)$} 

This moduli space can also be decomposed into five strata:
an open stratum $X_0$; a locally closed stratum that is the disjoint union of two irreducible
locally closed subsets $X_1$ and $X_2$, each of codimension $3$; a locally closed stratum that is
the disjoint union of two irreducible locally closed subsets $X_3$ and $X_4$, each of codimension $5$;
an irreducible locally closed stratum $X_5$ of codimension $7$ and a closed irreducible stratum
$X_6$ of codimension $9$.
For some of these sets we have concrete geometric descriptions:
$X_1$ is a certain open subset inside a fibre bundle with fibre $\P^{20}$ and base $\N(3,4,3) \times \P^2$;
$X_3$ is an open subset of a fibre bundle with fibre $\P^{22}$ and base $\Hilb(2) \times \N(3,2,3)$;
$X_5$ is an open subset of a fibre bundle with fibre $\P^{24}$ and base $\P^2 \times \Hilb(2)$;
the closed stratum $X_6$ is isomorphic to the universal sextic in $\P^2 \times \P(\SS^6 V^*)$.
The classification of sheaves in $\M(6,2)$ is summarised in Table 2 below, which is organised
as Table 1.

\begin{table}[!hpt]{Table 2. Summary for $\M(6,2)$.}
\begin{center}
{\small
\begin{tabular}{|c|c|c|c|}
\hline \hline
{}
&
{\tiny cohomological conditions}
&
{\tiny classification of sheaves $\F$ giving points in $X_i$}
&
{\tiny codim.}
\\
\hline
$X_0$
&
\begin{tabular}{r}
{} \\
$\h^0(\F(-1))=0$ \\
$\h^1(\F)=0$\\
$\h^0(\F \tensor \Om^1(1))=0$ \\
{}
\end{tabular}
&
\begin{tabular}{c}
{} \\
$0 \lra 4\O(-2) \stackrel{\f}{\lra} 2\O(-1) \oplus 2\O \lra \F \lra 0$ \\
$\f$ is not equivalent to a morphism of any of the forms \\
${\ds
\left[
\ba{cccc}
\star & 0 & 0 & 0 \\
\star & \star & \star & \star \\
\star & \star & \star & \star \\
\star & \star & \star & \star
\ea
\right], \quad \left[
\ba{cccc}
\star & \star & 0 & 0 \\
\star & \star & 0 & 0 \\
\star & \star & \star & \star \\
\star & \star & \star & \star
\ea
\right], \quad \left[
\ba{cccc}
\star & \star & \star & 0 \\
\star & \star & \star & 0 \\
\star & \star & \star & 0 \\
\star & \star & \star & \star
\ea
\right]
}$ \\
{}
\end{tabular}
&
0
\\
\hline
$X_1$
&
\begin{tabular}{r}
$\h^0(\F(-1))=0$ \\
$\h^1(\F)=0$\\
$\h^0(\F \tensor \Om^1(1))=1$
\end{tabular}
&
\begin{tabular}{c}
{} \\
$0 \lra 4\O(-2) \oplus \O(-1) \stackrel{\f}{\lra} 3\O(-1) \oplus 2\O \lra \F \lra 0$ \\
$\f_{12}=0$, $\f_{11}$ and $\f_{22}$ are semi-stable as Kronecker modules \\
{}
\end{tabular}
&
3
\\
\hline
$X_2$
&
\begin{tabular}{r}
{} \\
$\h^0(\F(-1))=0$ \\
$\h^1(\F)=1$\\
$\h^0(\F \tensor \Om^1(1))=1$ \\
{}
\end{tabular}
&
\begin{tabular}{c}
{} \\
$0 \lra \O(-3) \oplus \O(-2) \oplus \O(-1) \stackrel{\f}{\lra} 3\O \lra \F \lra 0$ \\
$\f$ is not equivalent to a morphism of any of the forms \\
${\ds
\left[
\ba{ccc}
\star & \star & \star \\
\star & \star & 0 \\
\star & \star & 0
\ea
\right], \quad \left[
\ba{ccc}
\star & \star & \star \\
\star & 0 & \star \\
\star & 0 & \star
\ea
\right], \quad \left[
\ba{ccc}
\star & \star & \star \\
\star & \star & \star \\
\star & 0 & 0
\ea
\right]
}$ \\
{}
\end{tabular}
&
3
\\
\hline
$X_{3}$
&
\begin{tabular}{r}
{} \\
$\h^0(\F(-1))=0$ \\
$\h^1(\F)=1$\\
$\h^0(\F \tensor \Om^1(1))=2$ \\
{}
\end{tabular}
&
\begin{tabular}{c}
{} \\
$0 \to \O(-3) \oplus \O(-2) \oplus 2\O(-1) \stackrel{\f}{\to} \O (-1) \oplus 3\O \to \F \to 0$ \\
$\f_{13}=0$, $\f_{12} \neq 0$ and does not divide $\f_{11}$ \\
$\f_{23}$ has linearly independent maximal minors \\
{}
\end{tabular}
&
5
\\
\hline
$X_{4}$
&
\begin{tabular}{r}
{} \\
$\h^0(\F(-1))=1$ \\
$\h^1(\F)=1$\\
$\h^0(\F \tensor \Om^1(1))=3$ \\
{}
\end{tabular}
&
\begin{tabular}{c}
{} \\
$0 \lra \O(-3) \oplus 2\O(-2) \stackrel{\f}{\lra} 2\O(-1) \oplus \O(1) \lra \F \to 0$ \\
$\f$ is not equivalent to a morphism of any of the forms \\
${\ds
\left[
\ba{ccc}
\star & 0 & 0 \\
\star & \star & \star \\
\star & \star & \star
\ea
\right], \left[
\ba{ccc}
\star & \star & 0 \\
\star & \star & 0 \\
\star & \star & \star
\ea
\right], \left[
\ba{ccc}
0 & 0 & \star \\
\star & \star & \star \\
\star & \star & \star
\ea
\right], \left[
\ba{ccc}
0 & \star & \star \\
0 & \star & \star \\
\star & \star & \star
\ea
\right]
}$ \\
{}
\end{tabular}
&
5
\\
\hline
$X_{5}$
&
\begin{tabular}{r}
{} \\
$\h^0(\F(-1))=1$ \\
$\h^1(\F)=2$\\
$\h^0(\F \tensor \Om^1(1))=4$ \\
{}
\end{tabular}
&
\begin{tabular}{c}
{} \\
$0 \lra 2\O(-3) \oplus \O(-1) \stackrel{\f}{\lra} \O(-2) \oplus \O \oplus \O(1) \lra \F \lra 0$ \\
$\f_{11}$ has linearly independent entries \\
$\f_{22} \neq 0$ and does not divide $\f_{32}$ \\
{}
\end{tabular}
&
7
\\
\hline
$X_6$
&
\begin{tabular}{r}
$\h^0(\F(-1))=2$ \\
$\h^1(\F)=3$ \\
$\h^0(\F \tensor \Om^1(1))=6$
\end{tabular}
&
\begin{tabular}{c}
{} \\
$0 \lra \O(-4) \oplus \O \lra 2\O(1) \lra \F \lra 0$ \\
$\f_{12}$ has linearly independent entries \\
{}
\end{tabular}
&
9
\\
\hline \hline
\end{tabular}
}
\end{center}
\end{table}

\subsection{The moduli space $\M(6,3)$} 

Here we have seven strata.
The open stratum $X_0$ is isomorphic to an open subset of $\N(6,3,3)$.
The locally closed stratum $X_1$ has codimension $1$ and is birational to $\P^{36}$.
The codimension $4$ stratum is the union of three irreducible locally closed subsets
$X_2$, $X_3$, $X_3^\D$.
Here $X_2$ is an open subset of a fibre bundle over $\N(3,3,2) \times \N(3,2,3)$
with fibre $\P^{21}$ and $X_3^\st$ isomorphic to an open subset
of a fibre bundle over $\N(3,3,4)$ with fibre $\P^{21}$.
The open subset $X_4^\st$ of the locally closed stratum $X_4$ of codimension $5$
is isomorphic to an open subset of a tower of bundles with fibre $\P^{21}$
and base a fibre bundle over $\P^5$ with fibre $\P^6$.
The locally closed stratum $X_5$ of codimension $6$
is isomorphic to an open subset of a fibre bundle over $\Hilb(2) \times \Hilb(2)$ with fibre $\P^{23}$.
The locally closed stratum $X_6$ is an open subset of a fibre bundle over $\P^2 \times \P^2$
with fibre $\P^{25}$ and has codimension $8$.
Finally, we have a closed stratum $X_7$ consisting of all sheaves of the form $\O_C(2)$
for $C \subset \P^2$ a sextic curve. Thus $X_7 \isom \P^{27}$.
The map $W_0 \to X_0$ is a good quotient map.
The map $W_1 \to X_1$ is a categorical quotient map.
The maps $W_3^\st \to X_3^\st$, $(W_3^\D)^\st \to (X_3^\D)^\st$, $W_4^\st \to X_4^\st$,
$W_i \to X_i$, $i = 2, 5, 6, 7$, are geometric quotient maps.

\begin{table}[!hpt]{Table 3. Summary for $\M(6,3)$.}
\begin{center}
{\small
\begin{tabular}{|c|c|c|c|}
\hline \hline
{}
&
{\tiny cohomological conditions}
&
{\tiny classification of sheaves $\F$ giving points in $X_i$}
&
{\tiny codim.}
\\
\hline
$X_0$
&
\begin{tabular}{r}
$\h^0(\F(-1))=0$ \\
$\h^1(\F)=0$\\
$\h^0(\F \tensor \Om^1(1))=0$
\end{tabular}
&
\begin{tabular}{c}
{} \\
$0 \lra 3\O(-2) \stackrel{\f}{\lra} 3\O \lra \F \lra 0$ \\
{}
\end{tabular}
&
$0$
\\
\hline
$X_1$
&
\begin{tabular}{r}
$\h^0(\F(-1))=0$ \\
$\h^1(\F)=0$\\
$\h^0(\F \tensor \Om^1(1))=1$
\end{tabular}
&
\begin{tabular}{c}
{} \\
$0 \lra 3\O(-2) \oplus \O(-1) \stackrel{\f}{\lra} \O(-1) \oplus 3\O \lra \F \lra 0$ \\
$\f_{12}=0$ \\
$\f$ is not equivalent to a morphism of any of the forms \\
${\ds
\left[
\ba{cccc}
\star & 0 & 0 & 0 \\
\star & \star & \star & \star \\
\star & \star & \star & \star \\
\star & \star & \star & \star
\ea
\right], \quad \left[
\ba{cccc}
\star & \star & 0 & 0 \\
\star & \star & 0 & 0 \\
\star & \star & \star & \star \\
\star & \star & \star & \star
\ea
\right], \quad \left[
\ba{cccc}
\star & \star & \star & 0 \\
\star & \star & \star & 0 \\
\star & \star & \star & 0 \\
\star & \star & \star & \star
\ea
\right]
}$ \\
{}
\end{tabular}
&
$1$
\\
\hline
$X_2$
&
\begin{tabular}{r}
$\h^0(\F(-1))=0$ \\
$\h^1(\F)=0$\\
$\h^0(\F \tensor \Om^1(1))=2$
\end{tabular}
&
\begin{tabular}{c}
{} \\
$0 \lra 3\O(-2) \oplus 2\O(-1) \stackrel{\f}{\lra} 2\O(-1) \oplus 3\O \lra \F \lra 0$ \\
$\f_{12}=0$ \\
$\f_{11}$ and $\f_{22}$ are semi-stable as Kronecker modules \\
{}
\end{tabular}
&
$4$
\\
\hline
$X_3$
&
\begin{tabular}{r}
$\h^0(\F(-1))=0$ \\
$\h^1(\F)=1$\\
$\h^0(\F \tensor \Om^1(1))=3$
\end{tabular}
&
\begin{tabular}{c}
{} \\
$0 \lra \O(-3) \oplus 3\O(-1) \stackrel{\f}{\lra} 4\O \lra \F \lra 0$ \\
$\f_{12}$ is semi-stable as a Kronecker module \\
{}
\end{tabular}
&
$4$
\\
\hline
${\ds X_3^\D}$
&
\begin{tabular}{r}
$\h^0(\F(-1))=1$ \\
$\h^1(\F)=0$\\
$\h^0(\F \tensor \Om^1(1))=3$
\end{tabular}
&
\begin{tabular}{c}
{} \\
$0 \lra 4\O(-2) \stackrel{\f}{\lra} 3\O(-1) \oplus \O(1) \lra \F \lra 0$ \\
$\f_{11}$ is semi-stable as a Kronecker module \\
{}
\end{tabular}
&
$4$
\\
\hline
$X_4$
&
\begin{tabular}{r}
$\h^0(\F(-1))=1$ \\
$\h^1(\F)=1$\\
$\h^0(\F \tensor \Om^1(1))=3$
\end{tabular}
&
\begin{tabular}{c}
{} \\
$0 \lra \O(-3) \oplus \O(-2) \stackrel{\f}{\lra} \O \oplus \O(1) \lra \F \lra 0$ \\
$\f_{12} \neq 0$ \\
{}
\end{tabular}
&
$5$
\\
\hline
$X_5$
&
\begin{tabular}{r}
$\h^0(\F(-1))=1$ \\
$\h^1(\F)=1$ \\
$\h^0(\F \tensor \Om^1(1))=4$
\end{tabular}
&
\begin{tabular}{c}
{} \\
$0 \to \O(-3) \oplus \O(-2) \oplus \O(-1) \stackrel{\f}{\to} \O(-1) \oplus \O \oplus \O(1) \to \F \to 0$ \\
$\f_{13}=0$, $\f_{12} \neq 0$, $\f_{23} \neq 0$, $\f_{12} \nmid \f_{11}$, $\f_{23} \nmid  \f_{33}$\\
{}
\end{tabular}
&
$6$
\\
\hline
$X_6$
&
\begin{tabular}{r}
$\h^0(\F(-1))=2$ \\
$\h^1(\F)=2$ \\
$\h^0(\F \tensor \Om^1(1))=6$
\end{tabular}
&
\begin{tabular}{c}
{} \\
$0 \lra 2\O(-3) \oplus \O \stackrel{\f}{\lra} \O(-2) \oplus 2\O(1) \lra \F \lra 0$ \\
$\f_{11}$ has linearly independent entries \\
$\f_{22}$ has linearly independent entries \\
{}
\end{tabular}
&
$8$
\\
\hline
$X_7$
&
\begin{tabular}{r}
$\h^0(\F(-1))=3$ \\
$\h^1(\F)=3$ \\
$\h^0(\F \tensor \Om^1(1))=8$
\end{tabular}
&
\begin{tabular}{c}
{} \\
$0 \lra \O(-4) \stackrel{\f}{\lra} \O(2) \lra \F \lra 0$ \\
{}
\end{tabular}
&
$10$
\\
\hline \hline
\end{tabular}
}
\end{center}
\end{table}

\subsection{The moduli space $\M(6,0)$} 

Here we have five strata: $X_0$, $X_1$, $X_2$, $X_3 \cup X_3^\D$ and $X_4$,
of codimensions given in Table 4 below.
The map $W_0 \to X_0$ is a good quotient map.
The maps $W_1 \to X_1$ and $W_2 \to X_2$ are categorical quotient maps away from
the points of the form $[\O_{C_1} \oplus \O_{C_2}]$, where $C_1$, $C_2$ are cubic curves.
The maps $W_3 \to X_3$, $W_3^\D \to X_3^\D$ and $W_4 \to X_4$ are geometric
quotient maps away from properly semi-stable points,
i.e. points of the form $[\O_L(-1) \oplus \O_Q(1)]$, where $L$ is a line and $Q$ is a quintic curve.
Thus $X_3^{\st}$ and $(X_3^\D)^\st$ are isomorphic to the open subset of $\Hilb(6,3)$
of pairs $(C,Z)$, where $C$ is a sextic curve, $Z \subset C$ is a zero-dimensional subscheme
of length $3$ that is not contained in a line.
Moreover, $X_4^\st$ is isomorphic to the locally closed subscheme of $\Hilb(6,3)$
given by the condition that $Z$ be contained in a line $L$ that is not a component of $C$.

\begin{table}[!hpt]{Table 4. Summary for $\M(6,0)$.}
\begin{center}
{\small
\begin{tabular}{|c|c|c|c|}
\hline \hline
{}
&
{\tiny cohomological conditions}
&
{\tiny classification of sheaves $\F$ giving points in $X_i$}
&
{\tiny codim.}
\\
\hline
$X_0$
&
\begin{tabular}{r}
$\h^0(\F(-1))=0$ \\
$\h^1(\F)=0$\\
$\h^1(\F(1))=0$
\end{tabular}
&
\begin{tabular}{c}
{} \\
$0 \lra 6\O(-2) \stackrel{\f}{\lra} 6\O(-1) \lra \F \lra 0$ \\
{}
\end{tabular}
&
0
\\
\hline
$X_1$
&
\begin{tabular}{r}
$\h^0(\F(-1))=0$ \\
$\h^1(\F)=1$\\
$\h^1(\F(1))=0$
\end{tabular}
&
\begin{tabular}{c}
{} \\
$0 \lra \O(-3) \oplus 3\O(-2) \stackrel{\f}{\lra} 3\O(-1) \oplus \O \lra \F \lra 0$ \\
$\f_{12}$ is semi-stable as a Kronecker module \\
{}
\end{tabular}
&
1
\\
\hline
$X_2$
&
\begin{tabular}{r}
$\h^0(\F(-1))=0$ \\
$\h^1(\F)=2$ \\
$\h^1(\F(1))=0$
\end{tabular}
&
\begin{tabular}{c}
{} \\
$0 \to 2\O(-3) \oplus \O(-2) \oplus \O(-1) \stackrel{\f}{\to} \O(-2) \oplus \O(-1) \oplus 2\O \to \F \to 0$ \\
where $\f$ has one of the following forms: \\
${\ds \left[
\ba{cccc}
0 & 0 & 1 & 0 \\
0 & 0 & 0 & 1 \\
f_{11} & f_{12} & 0 & 0 \\
f_{21} & f_{22} & 0 & 0
\ea
\right], \quad \left[
\ba{cccc}
\ell_1 & \ell_2 & 0 & 0 \\
0 & 0 & 0 & 1 \\
f_{11} & f_{12} & q_1 & 0 \\
f_{21} & f_{22} & q_2 & 0
\ea
\right]}$, \\
${\ds \left[
\ba{cccc}
0 & 0 & 1 & 0 \\
q_1 & q_2 & 0 & 0 \\
f_{11} & f_{12} & 0 & \ell_1 \\
f_{21} & f_{22} & 0 & \ell_2
\ea
\right], \qquad \left[
\ba{cccc}
\ell_1 & \ell_2 & 0 & 0 \\
p_1 & p_2 & \ell & 0 \\
f_{11} & f_{12} & p_1' & \ell_1' \\
f_{21} & f_{22} & p_2' & \ell_2'
\ea
\right]}$, \\
where $q_1, q_2$ are linearly independent, $\ell \neq 0$, \\
$\ell_1, \ell_2$ are linearly independent and the same for $\ell_1', \ell_2'$
\\
{}
\end{tabular}
&
4
\\
\hline
$X_3$
&
\begin{tabular}{r}
$\h^0(\F(-1))=0$ \\
$\h^1(\F)=3$\\
$\h^1(\F(1))=1$
\end{tabular}
&
\begin{tabular}{c}
{} \\
$0 \lra \O(-4) \oplus 2\O(-1) \stackrel{\f}{\lra} 3\O \lra \F \lra 0$ \\
$\f_{12}$ has linearly independent maximal minors \\
{}
\end{tabular}
&
7
\\
\hline
$X_3^\D$
&
\begin{tabular}{r}
$\h^0(\F(-1))=1$ \\
$\h^1(\F)=3$\\
$\h^1(\F(1))=0$
\end{tabular}
&
\begin{tabular}{c}
{} \\
$0 \lra 3\O(-3) \stackrel{\f}{\lra} 2\O (-2) \oplus \O(1) \lra \F \lra 0$ \\
$\f_{11}$ has linearly independent maximal minors \\
{}
\end{tabular}
&
7
\\
\hline
$X_4$
&
\begin{tabular}{r}
$\h^0(\F(-1))=1$ \\
$\h^1(\F)=3$ \\
$\h^1(\F(1))=1$ 
\end{tabular}
&
\begin{tabular}{c}
{} \\
$0 \lra \O(-4) \oplus \O(-2) \stackrel{\f}{\lra} \O(-1) \oplus \O(1) \lra \F \lra 0$ \\
$\f_{12} \neq 0$ \\
{}
\end{tabular}
&
8
\\
\hline \hline
\end{tabular}
}
\end{center}
\end{table}



\section{Preliminaries}

\subsection{The Beilinson monad and spectral sequences} 

In this subsection $\F$ will be a coherent sheaf on $\P^2$ with support of dimension $1$.
The $\EE^1$-term of the Beilinson spectral sequence I converging to $\F$ has display diagram
\[
\tag{2.1.1}
\xymatrix
{
\H^1(\F(-2)) \tensor \O(-1) \ar[r]^-{\f_1} & \H^1(\F(-1)) \tensor \Om^1(1) \ar[r]^-{\f_2} & \H^1(\F) \tensor \O \\
\H^0(\F(-2)) \tensor \O(-1) \ar[r]^-{\f_3} & \H^0(\F(-1)) \tensor \Om^1(1) \ar[r]^-{\f_4} & \H^0(\F) \tensor \O
}.
\]
The spectral sequence degenerates at $\EE^3$, which shows that $\f_2$ is surjective and that we have the exact
sequenes
\[
\tag{2.1.2}
0 \to \H^0(\F(-2)) \tensor \O(-1) \stackrel{\f_3}{\to} \H^0(\F(-1)) \tensor \Om^1(1) \stackrel{\f_4}{\to} \H^0(\F) \tensor \O
\to \Coker(\f_4) \to 0,
\]
\[
\tag{2.1.3}
0 \lra \Ker(\f_1) \stackrel{\f_5}{\lra} \Coker(\f_4) \lra \F \lra \Ker(\f_2)/\Im(\f_1) \lra 0.
\]
The $\EE^1$-term of the Beilinson spectral sequence II converging to $\F$ has display diagram
\[
\tag{2.1.4}
\xymatrix
{
\H^1(\F(-1)) \tensor \O(-2) \ar[r]^-{\f_1} & \H^1(\F \tensor \Om^1(1)) \tensor \O(-1) \ar[r]^-{\f_2} & \H^1(\F) \tensor \O \\
\H^0(\F(-1)) \tensor \O(-2) \ar[r]^-{\f_3} & \H^0(\F \tensor \Om^1(1)) \tensor \O(-1) \ar[r]^-{\f_4} & \H^0(\F) \tensor \O
}.
\]
As above, this spectral sequence degenerates at $\EE^3$ and yields the exact sequences
\[
\tag{2.1.5}
0 \to \H^0(\F(-1)) \tensor \O(-2) \stackrel{\f_3}{\to} \H^0(\F \tensor \Om^1(1)) \tensor \O(-1)
\stackrel{\f_4}{\to} \H^0(\F) \tensor \O \to \Coker(\f_4) \to 0,
\]
\[
\tag{2.1.6}
0 \lra \Ker(\f_1) \stackrel{\f_5}{\lra} \Coker(\f_4) \lra \F \lra \Ker(\f_2)/\Im(\f_1) \lra 0.
\]
The Beilinson free monad associated to $\F$ is a sequence
\[
\tag{2.1.7}
0 \lra \CC^{-2} \lra \CC^{-1} \lra \CC^0 \lra \CC^1 \lra \CC^2 \lra 0,
\]
\[
\CC^p = \bigoplus_{i+j=p} \H^j (\F \tensor \Om^{-i}(-i)) \tensor \O(i),
\]
that is exact, except at $\CC^0$, where the cohomology is $\F$.
Note that $\CC^2=0$ because $\F$ is assumed to have dimension $1$.
The maps
\[
\H^0(\F \tensor \Om^{-i}(-i)) \tensor \O(i) \lra \H^1(\F \tensor \Om^{-i}(-i)) \tensor \O(i),
\]
$i = 0, -1, -2$, occurring in the monad are zero, cf., for instance, \cite{maican-duality}, lemma 1.

\subsection{Cohomology bounds} 

\begin{prop} \hfill
\label{2.2.1}
\begin{enumerate}
\item[(i)] Let $\F$ give a point in $\M(r,\chi)$, where $0 \le \chi < r$. Assume that $\h^1(\F) > 0$.
Then $\h^1(\F(1)) > 2 \h^1(\F) - \h^1(\F(-1))$.
\item[(ii)] Let $\F$ give a point in $\M(r,\chi)$, where $0 < \chi \le r$. Assume that $\h^0(\F(-1)) > 0$.
Then $\h^0(\F(-2)) > 2 \h^0(\F(-1)) - \h^0(\F)$.
\end{enumerate}
\end{prop}

\begin{proof}
Part (ii) is equivalent to (i) by duality, so we concentrate on (i).
Write $p = \h^1(\F)$, $q= \h^0(\F(-1))$, $m = \h^0(\F \tensor \Om^1(1))$.
The Beilinson free monad (2.1.7)  for $\F$ takes the form
\begin{multline*}
0 \lra q\O(-2) \stackrel{\psi}{\lra} (q+r-\chi) \O(-2) \oplus m\O(-1) \lra \\
(m+r-2\chi)\O(-1) \oplus (p+\chi)\O \stackrel{\eta}{\lra} p\O \lra 0
\end{multline*}
and yields a resolution
\[
0 \lra (q+r-\chi)\O(-2) \oplus \Coker(\psi_{21}) \stackrel{\f}{\lra} \Ker(\eta_{11}) \oplus (p+\chi)\O \lra \F \lra 0
\]
in which $\f_{12}=0$. Since $\F$ maps surjectively to $\Coker(\f_{11})$ we have the inequality
\[
m+r-2\chi -p = \rank(\Ker(\eta_{11})) \le q+r-\chi.
\]
If the inequality is not strict, then $\Coker(\f_{11})$ has negative slope, contradicting the
semi-stability of $\F$. Thus $m < p+q+\chi$.
We have
\begin{align*}
\h^0(\F(1)) & = \h^0((p+\chi)\O(1)) + \h^0(\Ker(\eta_{11})(1)) - \h^0(\Coker(\psi_{21})(1)) \\
& \ge \h^0((p+\chi)\O(1)) - \h^0(\Coker(\psi_{21})(1)) \\
& = 3p + 3\chi - m \\
& > 2p + 2\chi - q, \\
\h^1(\F(1)) & = \h^0(\F(1)) - r - \chi > 2p+ \chi - q - r = 2\h^1(\F) - \h^1(\F(-1)).
\qedhere
\end{align*}
\end{proof}

\begin{corollary}
\label{2.2.2}
There are no sheaves $\F$ giving points
\begin{enumerate}
\item[(i)] in $\M(6,1)$ and satisfying
$\h^0(\F(-1)) \le 1$, $\h^1(\F) \ge 3$, $\h^1(\F(1))=0$;
\item[(ii)] in $\M(6,1)$ and satisfying
$\h^0(\F(-1)) = 1$, $\h^1(\F) = 1$;
\item[(iii)] in $\M(6,1)$ and satisfying
$\h^0(\F(-1)) = 2$, $\h^1(\F(1)) = 0$;
\item[(iv)] in $\M(6,2)$ and satisfying
$\h^0(\F(-1)) \le 1$, $\h^1(\F) \ge 3$, $\h^1(\F(1))=0$;
\item[(v)] in $\M(6,2)$ and satisfying
$\h^0(\F(-1)) = 0$, $\h^1(\F) = 2$, $\h^1(\F(1))=0$;
\item[(vi)] in $\M(6,3)$ and satisfying
$\h^0(\F(-1)) \le 1$, $\h^1(\F) \ge 2$, $\h^1(\F(1))=0$;
\item[(vii)] in $\M(6,0)$ and satisfying
$\h^0(\F(-1)) = 0$, $\h^1(\F) \ge 3$, $\h^1(\F(1))=0$.
\end{enumerate}
\end{corollary}

\begin{proof}
Let $\F$ give a point in $\M(6,1)$. According to 2.1.3 \cite{drezet-maican}, $\h^0(\F(-2))=0$.
In view of \ref{2.2.1}(ii) we have $\h^0(\F) > 2\h^0(\F(-1))$.
This proves (ii).
Assume now that $\F$ satisfies the cohomological conditions from (iii).
Then $\h^1(\F) = \h^0(\F) -1 \ge 4$.
On the other hand, by \ref{2.2.1}(i), we have $7 = \h^1(\F(-1)) > 2\h^1(\F)$.
This yields a contradiction and proves (iii). All other parts of the corollary are direct
applications of \ref{2.2.1}(i).
\end{proof}

\subsection{Stability criteria} 

\begin{prop}
\label{2.3.1}
Let $n$ be a positive integer and let $d_1 \le \cdots \le d_n$, $e_1 \le \cdots \le e_n$
be integers satisfying the relations
\begin{align*}
\tag{i}
e_1 - d_1 & \ge e_2 + \cdots + e_n - d_2 - \cdots - d_n, \\
\tag{ii}
e_1 + d_1 & \le \frac{e_2^2 + \cdots + e_n^2 - d_2^2 - \cdots - d_n^2}{e_2 + \cdots + e_n - d_2 - \cdots - d_n}.
\end{align*}
Let $\F$ be a sheaf on $\P^2$ having resolution
\[
0 \lra \O(d_1) \oplus \cdots \oplus \O(d_n) \stackrel{\f}{\lra} \O(e_1) \oplus \cdots \oplus \O(e_n) \lra \F \lra 0.
\]
Assume that the maximal minors of the restriction of $\f$ to $\O(d_2) \oplus \cdots \oplus \O(d_n)$
have no common factor and that none of them has degree zero.
Then $\F$ is stable, unless the ratio
\[
r = \frac{e_1^2 + \cdots + e_n^2 - d_1^2 - \cdots - d_n^2}{e_1 + \cdots + e_n - d_1 - \cdots - d_n}
\]
is an integer and $\F$ has a subsheaf $\S$ given by a resolution
\[
0 \lra \O(d_1) \lra \O(r-d_1) \lra \S \lra 0.
\]
In this case $\pp(\S) = \pp(\F)$ and $\F$ is properly semi-stable.
Note that condition (ii) can be replaced by the requirement that $e_i \ge d_i$ for $2 \le i \le n$.
\end{prop}

\begin{proof}
Let $C \subset \P^2$ be the curve given by the equation $\det(\f)=0$.
Its degree is $d = e_1 + \cdots + e_n - d_1 - \cdots - d_n$.
Let $\psi$ denote the restriction of $\f$ to $\O(d_2) \oplus \cdots \oplus \O(d_n)$
and let $\z_i$ be the maximal minor of the matrix representing $\psi$ obtained by deleting the $i$-th row.
We have an exact sequence
\[
0 \lra \O(d_2) \oplus \cdots \oplus \O(d_n) \stackrel{\psi}{\lra} \O(e_1) \oplus \cdots \oplus \O(e_n)
\stackrel{\z}{\lra} \O(e) \lra \CC \lra 0,
\]
\[
\z = \left[
\ba{cccc}
\z_1 & -\z_2 & \cdots & (-1)^{n+1} \z_n
\ea
\right], \qquad e = d+ d_1.
\]
The Hilbert polynomial of $\CC$ is a constant,
namely ${\ds \frac{d^2}{2} + dd_1 + \sum_{i=1}^n \frac{d_i^2 - e_i^2}{2}}$, showing that $\CC$ is the structure
sheaf of a zero-dimensional scheme $Z \subset \P^2$, that $\Coker(\psi) \isom \I_Z(e)$ and $\F \isom \J_Z(e)$,
where $\J_Z \subset \O_C$ is the ideal sheaf of $Z$ in $C$.
Clearly $\F$ has no zero-dimensional torsion.
Let $\S \subset \F$ be a subsheaf of multiplicity at most $d-1$.
According to \cite{maican}, lemma 6.7, there is a sheaf $\A$
such that $\S \subset \A \subset \O_C(e)$, $\A/\S$ is supported on finitely many points
and $\O_C(e)/\A \isom \O_S(e)$ for a curve $S \subset \P^2$ of degree $s$, $1 \le s \le d-1$.
We have the relations
\begin{align*}
\PP_{\S}(m) & = \PP_{\O_C(e)}(m) - \PP_{\O_S(e)}(m) - \h^0(\A/\S) \\
& = dm + de - \frac{d(d-3)}{2} - sm - se + \frac{s(s-3)}{2} - \h^0(\A/\S), \\
\pp(\S) & = e + \frac{3}{2} - \frac{d+s}{2} - \frac{\h^0(\A/\S)}{d-s}, \\
\PP_{\F}(m) & = dm + \frac{3d}{2} + \sum_{i=1}^n \frac{e_i^2 - d_i^2}{2}, \\
\pp(\F) & = \frac{3}{2} + \sum_{i=1}^n \frac{e_i^2 - d_i^2}{2d}.
\end{align*}
In order to show that $\F$ is semi-stable we will prove that $\pp(\S) \le \pp(\F)$.
This is equivalent to the inequality
\[
d_1 + \frac{d}{2} + \sum_{i=1}^n \frac{d_i^2 - e_i^2}{2d} \le \frac{s}{2} + \frac{\h^0(\A/\S)}{d-s}.
\]
Assume that
\[
\frac{s}{2} < d_1 + \frac{d}{2} + \sum_{i=1}^n \frac{d_i^2 - e_i^2}{2d} \quad \text{and} \quad
\h^0(\A/\S) \le (d-s) \Big( d_1 + \frac{d}{2} - \frac{s}{2} + \sum_{i=1}^n \frac{d_i^2 - e_i^2}{2d} \Big).
\]
We have a commutative diagram
\[
\xymatrix
{
0 \ar[r] & \A/\S \egal[d] \ar[r] & \O_C(e)/\S \surj[d] \ar[r] & \O_S(e) \surj[d] \ar[r] & 0 \\
& \A/\S \ar[r] & \O_Z \ar[r] & \O_Y \ar[r] & 0
}
\]
in which $Y$ is a subscheme of $Z$ of length at least
\[
\frac{d^2}{2} + dd_1 + \sum_{i=1}^n \frac{d_i^2 - e_i^2}{2} -
(d-s) \Big( d_1 + \frac{d}{2} - \frac{s}{2} + \sum_{i=1}^n \frac{d_i^2 - e_i^2}{2d} \Big)
= s \Big( d + d_1 - \frac{s}{2} + \sum_{i=1}^n \frac{d_i^2 - e_i^2}{2d} \Big).
\]
We claim that $\length(Y) > s \deg(\z_1) = s(d-e_1 + d_1)$.
This follows from the equivalent inequalities
\begin{align*}
d + d_1 - \frac{s}{2} + \sum_{i=1}^n \frac{d_i^2 - e_i^2}{2d} & > d-e_1 + d_1, \\
e_1 + \sum_{i=1}^n \frac{d_i^2 - e_i^2}{2d} & > \frac{s}{2},
\end{align*}
which follow from the inequality
\[
e_1 + \sum_{i=1}^n \frac{d_i^2 - e_i^2}{2d} \ge
d_1 + \frac{d}{2} + \sum_{i=1}^n \frac{d_i^2 - e_i^2}{2d}.
\]
The latter is equivalent to condition (i) from the hypothesis.
This proves the claim.
Since $Y$ is a subscheme of $S$ and also of the curve given by the equation $\z_1 = 0$,
we can apply B\'ezout's Theorem to deduce that
$S$ and the curve given by the equation $\z_1 =0$ have a common component.
Since $\gcd(\z_1, \ldots, \z_n)=1$, we may perform elementary column operations on the matrix
representing $\f$ to ensure that $\z_1$ is irreducible.
Thus $\z_1$ divides the equation defining $S$. In particular, $\deg(\z_1) \le s$.
It follows that
\begin{align*}
d - e_1 + d_1 = \deg(\z_1) & < 2 d_1 + d + \sum_{i=1}^n \frac{d_i^2 - e_i^2}{d}, \\
\sum_{i=1}^n (e_i^2 - d_i^2) & < d(d_1+e_1) = e_1^2 - d_1^2 + (d_1 + e_1) \sum_{i=2}^n (e_i - d_i), \\
\sum_{i=2}^n (e_i^2 - d_i^2) & < (d_1 + e_1) \sum_{i=2}^n (e_i - d_i).
\end{align*}
The last inequality contradicts condition (ii) from the hypothesis.
The above discussion shows that $\pp(\S) < \pp(\F)$ unless $\S= \A$ and
\[
\frac{s}{2} = d_1 + \frac{d}{2} + \sum_{i=1}^n \frac{d_i^2 - e_i^2}{2d},
\]
in which case $\pp(\S) = \pp(\F)$ and $\F$ is semi-stable but not stable.
Applying the snake lemma to the commutative diagram
\[
\xymatrix
{
0 \ar[r] & \O(e-d) \ar[r] \ar[d] & \O(e) \ar[r] \egal[d] & \O_C(e) \ar[r] \ar[d] & 0 \\
0 \ar[r] & \O(e-s) \ar[r] & \O(e) \ar[r] & \O_S(e) \ar[r] & 0
}
\]
yields the exact sequence
\[
0 \lra \O(e-d) \lra \O(e-s) \lra \S \lra 0.
\]
We have $e-d = d_1$, $e-s= r-d_1$.
\end{proof}

\begin{corollary}
\label{2.3.2}
Let $d_1 \le d_2 < e_1 \le e_2$ be integers satisfying the condition $e_1 - d_1 \ge e_2 - d_2$.
Let $\F$ be a sheaf on $\P^2$ having resolution
\[
0 \lra \O(d_1) \oplus \O(d_2) \stackrel{\f}{\lra} \O(e_1) \oplus \O(e_2) \lra \F \lra 0.
\]
Assume that $\f_{12}$ and $\f_{22}$ have no common factor. Then $\F$ is stable, unless
$e_1 - d_1 = e_2 - d_2$ and
${\ds
\f \sim \left[
\ba{cc}
0 & \star \\
\star & \star
\ea
\right]
}$,
in which case $\F$ is semi-stable but not stable.
\end{corollary}

\begin{proof}
According to the proposition above, $\F$ is stable unless the ratio
\[
r = \frac{e_1^2 + e_2^2 - d_1^2 - d_2^2}{e_1 + e_2 - d_1 - d_2}
\]
is an integer and $\F$ has a subsheaf $\S$ given by a certain resolution.
We have a commutative diagram
\[
\xymatrix
{
0 \ar[r] & \O(d_1) \ar[r] \ar[d]^-{\b} & \O(r-d_1) \ar[r] \ar[d]^-{\a} & \S \ar[r] \ar[d] & 0 \\
0 \ar[r] & \O(d_1) \oplus \O(d_2) \ar[r]^-{\f} & \O(e_1) \oplus \O(e_2) \ar[r] & \F \ar[r] & 0
}
\]
in which $\a$ and $\b$ are injective. Thus $r-d_1 \le e_2$, that is
\begin{align*}
e_1^2 + e_2^2 - d_1^2 - d_2^2 & \le (e_1 + e_2 -d_1 - d_2)(d_1 + e_2), \\
(e_1-d_2) (e_1 + d_2) & \le (e_1 - d_2) (d_1 + e_2), \\
e_1 + d_2 & \le d_1 + e_2.
\end{align*}
Thus $e_1-d_1 = e_2 - d_2$, $r-d_1 = e_2$ and $\f$ has the special form given above.
\end{proof}


\section{The moduli space $\M(6,1)$}

\subsection{Classification of sheaves}  

\begin{prop}
\label{3.1.1}
Every sheaf $\F$ giving a point in $\M(6,1)$ and satisfying the condition
$\h^1(\F)=0$ also satisfies the condition $\h^0(\F(-1))=0$.
These sheaves are precisely the sheaves having a resolution of the form
\[
0 \lra 5\O(-2) \stackrel{\f}{\lra} 4\O(-1) \oplus \O \lra \F \lra 0,
\]
where $\f_{11}$ is semi-stable as a Kronecker module.
\end{prop}

\begin{proof}
The statement follows by duality from 4.2 \cite{maican}.
\end{proof}

\begin{claim}
\label{3.1.2}
Consider an exact sequence of sheaves on $\P^2$
\[
0 \lra \O(-3) \oplus 2\O(-2) \stackrel{\f}{\lra} 2\O(-1) \oplus \O(1) \lra \F \lra 0,
\]
\[
\f = \left[
\ba{ccc}
q_1 & \ell_{11} & \ell_{12} \\
q_2 & \ell_{21} & \ell_{22} \\
f & g_1 & g_2
\ea
\right],
\]
where $\ell_{11} \ell_{22} - \ell_{12} \ell_{21} \neq 0$ and the images of
$q_1 \ell_{21} - q_2 \ell_{11}$ and $q_1 \ell_{22} - q_2 \ell_{12}$
in $\SS^3 V^*/(\ell_{11} \ell_{22} - \ell_{12} \ell_{21})V^*$ are linearly independent.
Then $\F$ gives a stable point in $\M(6,2)$.
\end{claim}

\begin{proof}
By hypothesis the maximal minors of the matrix
\[
\psi = \left[
\ba{ccc}
q_1 & \ell_{11} & \ell_{12} \\
q_2 & \ell_{21} & \ell_{22}
\ea
\right]
\]
cannot have a common quadratic factor.
If they have no common factor, then the claim follows
by duality from \ref{2.3.1}. Assume that they have a common linear factor.
Then $\Ker(\psi) \isom \O(-4)$ and $\Coker(\psi)$ is supported on a line $L$.
From the snake lemma we get an extension
\[
0 \lra \O_C(1) \lra \F \lra \Coker(\psi) \lra 0,
\]
where $C$ is a quintic curve. Because of the conditions on $\psi$ it is easy to check that
$\Coker(\psi)$ has zero-dimensional torsion of length at most $1$.
Assume that $\Coker(\psi)$ has no zero-dimensional torsion, i.e. $\Coker(\psi) \isom \O_L(1)$.
Let $\F' \subset \F$ be a non-zero subsheaf of multiplicity at most $5$.
Denote by $\CC$ its image in $\O_L(1)$ and put $\K= \F' \cap \O_C(1)$.
If $\CC=0$, then $\pp(\F') \le 0$ because $\O_C$ is stable.
Assume that $\CC \neq 0$, i.e. that $\CC$ has multiplicity $1$.
If $\K=0$ and $\F'$ destabilises $\F$, then $\F'\isom \O_L$ or $\F' \isom \O_L(1)$.
Both situations can be ruled out using diagrams analogous to diagram (8) at \ref{3.1.3} below.
Thus we may assume that $1 \le \mult(\K) \le 4$.
According to \cite{maican}, lemma 6.7, there is a sheaf $\A$ such that $\K \subset \A \subset \O_C(1)$,
$\A/\K$ is supported on finitely many points and $\O_C(1)/\A \isom \O_S(1)$
for a curve $S \subset \P^2$ of degree $s$, $1 \le s \le 4$.
Thus
\begin{align*}
\PP_{\F'}(m) & = \PP_{\K}(m) + \PP_{\CC}(m) \\
& = \PP_{\A}(m) - \h^0(\A/\K) + \PP_{\O_L(1)}(m) - \h^0(\O_L(1)/\CC) \\
& = (5-s)m + \frac{s^2-5s}{2} + m + 2 - \h^0(\A/\K) - \h^0(\O_L(1)/\CC), \\
\pp(\F') & = \frac{1}{6-s}\left( \frac{s^2-5s}{2} + 2 - \h^0(\A/\K) - \h^0(\O_L(1)/\CC) \right) \\
& \le \frac{s^2-5s+4}{2(6-s)} < \frac{1}{3} = \pp(\F).
\end{align*}
We see that in this case $\F$ is stable.
Assume next that $\Coker(\psi)$ has a zero-dimensional subsheaf $\TT$ of length $1$.
Let $\E$ be the preimage of $\TT$ in $\F$.
According to 3.1.5 \cite{mult_five}, $\E$ gives a point in $\M(5,1)$.
Let $\F'$ and $\CC$ be as above. If $\CC \subset \TT$, then $\F' \subset \E$, hence
$\pp(\F') \le \pp(\E) < \pp(\F)$. If $\CC$ is not a subsheaf of $\TT$, then we can estimate
$\pp(\F')$ as above concluding again that it is less than the slope of $\F$.
\end{proof}

\begin{prop}
\label{3.1.3}
The sheaves $\F$ giving points in $\M(6,1)$ and satisfying the conditions
$\h^0(\F(-1)) = 0$, $\h^1(\F) = 1$, $\h^1(\F(1)) = 0$
are precisely the sheaves having a resolution of the form
\[
\tag{i}
0 \lra \O(-3) \oplus 2\O(-2) \stackrel{\f}{\lra} \O(-1) \oplus 2\O \lra \F \lra 0,
\]
\[
\f= \left[
\ba{ccc}
q & \ell_1 & \ell_2 \\
f_1 & q_{11} & q_{12} \\
f_2 & q_{21} & q_{22}
\ea
\right],
\]
where $\f$ is not equivalent to a morphism represented by a matrix of one of the
following four forms:
\[
\f_1 = \left[
\ba{ccc}
\star & 0 & 0 \\
\star & \star & \star \\
\star & \star & \star
\ea
\right], \ \f_2= \left[
\ba{ccc}
\star & \star & 0 \\
\star & \star & 0 \\
\star & \star & \star
\ea
\right], \ \f_3 = \left[
\ba{ccc}
\star & \star & \star \\
\star & \star & \star \\
\star & 0 & 0
\ea
\right], \ \f_4 = \left[
\ba{ccc}
0 & 0 & \star \\
\star & \star & \star \\
\star & \star & \star
\ea
\right],
\]
or the sheaves having a resolution of the form
\[
\tag{ii}
0 \lra \O(-3) \oplus 2\O(-2) \oplus \O(-1) \stackrel{\f}{\lra} 2\O(-1) \oplus 2\O \lra \F \lra 0,
\]
\[
\f = \left[
\ba{cccc}
q_1 & \ell_{11} & \ell_{12} & 0 \\
q_2 & \ell_{21} & \ell_{22} & 0 \\
f_1 & q_{11} & q_{12} & \ell_1 \\
f_2 & q_{21} & q_{22} & \ell_2
\ea
\right],
\]
where $\ell_1, \ell_2$ are linearly independent one-forms,
$\ell_{11} \ell_{22} - \ell_{12} \ell_{21} \neq 0$ and the images of
$q_1 \ell_{21} - q_2 \ell_{11}$ and $q_1 \ell_{22} - q_2 \ell_{12}$
in $\SS^3 V^*/(\ell_{11} \ell_{22} - \ell_{12} \ell_{21})V^*$ are linearly independent.
\end{prop}

\begin{proof}
Let $\F$ give a point in $\M(6,1)$ and satisfy the above cohomological conditions.
Display diagram (2.1.1) for the Beilinson spectral sequence I converging to $\F(1)$ reads
\[
\xymatrix
{
5\O(-1) \ar[r]^-{\f_1} & \Om^1(1) & 0 \\
0 & 2\Om^1(1) \ar[r]^-{\f_4} & 7\O
}.
\]
Resolving $\Om^1(1)$ yields the exact sequence
\[
0 \lra \Ker(\f_1) \lra \O(-2) \oplus 5\O(-1) \stackrel{\sigma}{\lra} 3\O(-1) \lra \Coker(\f_1) \lra 0.
\]
Notice that $\F(1)$ maps surjectively to $\Coker(\f_1)$.
Thus $\rank(\sigma_{12})=3$, otherwise $\Coker(\f_1)$ would have positive rank
or would be isomorphic to $\O_L(-1)$ violating the semi-stability of $\F(1)$.
We have shown that $\Coker(\f_1)=0$ and $\Ker(\f_1) \isom \O(-2) \oplus 2\O(-1)$.
Combining the exact sequences (2.1.2) and (2.1.3) we obtain the resolution
\[
0 \lra \O(-2) \oplus 2\O(-1) \oplus 2\Om^1(1) \lra 7\O \lra \F(1) \lra 0,
\]
hence a resolution
\[
0 \lra \O(-2) \oplus 2\O(-1) \oplus 6\O \stackrel{\rho}{\lra} 7\O \oplus 2\O(1) \lra \F(1) \lra 0.
\]
Notice that $\rank(\rho_{13}) \ge 5$ otherwise $\F(1)$ would map surjectively to
the cokernel of a morphism $\O(-2) \oplus 2\O(-1) \to 3\O$, in violation of semi-stability.
Canceling $5\O$ and tensoring with $\O(-1)$ we arrive at the resolution
\[
0 \lra \O(-3) \oplus 2\O(-2) \oplus \O(-1) \stackrel{\f}{\lra} 2\O(-1) \oplus 2\O \lra \F \lra 0.
\]
From this we get resolution (i) or (ii), depending on whether $\f_{13} \neq 0$ or $\f_{13} =0$.

\medskip

\noi
Conversely, we assume that $\F$ has resolution (i) and we need to show that there are no
destabilising subsheaves $\E$. We argue by contradiction, i.e. we assume that there is such
a subsheaf $\E$. We may assume that $\E$ is semi-stable.
As $\h^0(\E) \le 2$, $\E$ gives a point in $\M(r,1)$ or $\M(r,2)$ for some $r$, $1 \le r \le 5$.
The cohomology groups $\H^0(\E(-1))$ and $\H^0(\E \tensor \Om^1(1))$ vanish
because the corresponding cohomology groups for $\F$ vanish.
From the description of $\M(r,1)$ and $\M(r,2)$, $1 \le r \le 5$, found in \cite{drezet-maican}
and \cite{mult_five}, we see that $\E$ may have one of the following resolutions:
\[
\tag{1}
0 \lra \O(-2) \lra \O  \lra \E \lra 0,
\]
\[
\tag{2}
0 \lra 2\O(-2) \lra \O(-1) \oplus \O \lra \E \lra 0,
\]
\[
\tag{3}
0 \lra 3\O(-2) \lra 2\O(-1) \oplus \O \lra \E \lra 0,
\]
\[
\tag{4}
0 \lra 2\O(-2) \lra 2\O \lra \E \lra 0,
\]
\[
\tag{5}
0 \lra 4\O(-2) \lra 3\O(-1) \oplus \O \lra \E \lra 0,
\]
\[
\tag{6}
0 \lra \O(-3) \oplus \O(-2) \lra 2\O \lra \E \lra 0,
\]
\[
\tag{7}
0 \lra 3\O(-2) \lra \O(-1) \oplus 2\O \lra \E \lra 0.
\]
Resolution (1) must fit into a commutative diagram
\[
\tag{8}
\xymatrix
{
0 \ar[r] & \O(-2) \ar[r]^-{\psi} \ar[d]^-{\b} & \O \ar[r] \ar[d]^-{\a} & \E \ar[r] \ar[d] & 0 \\
0 \ar[r] & \O(-3) \oplus 2\O(-2) \ar[r]^-{\f} & \O(-1) \oplus 2\O \ar[r] & \F \ar[r] & 0
}
\]
in which $\a$ is injective (being injective on global sections).
Thus $\b$ is injective, too, and $\f \sim \f_2$, contradicting our hypothesis on $\f$.
Similarly, every other resolution must fit into a commutative diagram in which
$\a$ and $\a(1)$ are injective on global sections.
This rules out resolution (7) because in that case $\a$ must be injective,
hence $\Ker(\b)=0$, which is absurd.
If $\E$ has resolution (5), then $\a$ is equivalent to a morphism
represented by a matrix having one of the following two forms:
\[
\left[
\ba{cccc}
1 & 0 & 0 & 0 \\
0 & u_1 & u_2 & 0 \\
0 & 0 & 0 & 1
\ea
\right] \quad \text{or} \quad \left[
\ba{cccc}
0 & 0 & 0 & 0 \\
u_1 & u_2 & u_3 & 0 \\
0 & 0 & 0 & 1
\ea
\right],
\]
where $u_1, u_2, u_3$ are linearly independent one-forms.
In the first case $\Ker(\b) \isom \O(-2)$, in the second case $\Ker(\b) \isom \Om^1$.
Both situations are absurd.
Assume that $\E$ has resolution (3).
Since $\b$ cannot be injective, we see that $\a$ is equivalent to a morphism
represented by a matrix of the form
\[
\left[
\ba{ccc}
0 & 0 & 0 \\
u_1 & u_2 & 0 \\
0 & 0 & 1
\ea
\right],
\]
hence $\Ker(\a) \isom \O(-2)$, hence $\f \sim \f_1$, which is a contradiction.
For resolutions (2), (4) and (6) $\a$ and $\b$ must be injective and we get
the contradictory conclusions that $\f \sim \f_3$, $\f \sim \f_1$, or $\f \sim \f_4$.

\medskip

\noi
Assume now that $\F$ has resolution (ii). The sheaf $\G= \F^\D(1)$ is the cokernel of the transpose of $\f$.
From the snake lemma we have an extension
\[
0 \lra \G' \lra \G \lra \C_x \lra 0,
\]
where $\G'$ is the cokernel of a morphism $\psi \colon \O(-3) \oplus 2\O(-1) \to 2\O \oplus \O(1)$,
\[
\psi = \left[
\ba{ccc}
\star & \ell_{22} & \ell_{12} \\
\star & \ell_{21} & \ell_{11} \\
\star & q_2 & q_1
\ea
\right].
\]
From \ref{3.1.2} we know that $\G'$ gives a stable point in $\M(6,4)$.
It is now straightforward to check that any destabilising subsheaf $\E$ of $\G$
must give a point in $\M(1,1)$ or $\M(2,2)$.
The existence of such sheaves can be ruled out as above using diagrams
analogous to diagram (8).
\end{proof}

\begin{claim}
\label{3.1.4}
Let $\F$ be a sheaf having a resolution
\[
0 \lra \O(-4) \oplus 2\O(-1) \stackrel{\psi}{\lra} 3\O \lra \F \lra 0
\]
in which $\psi_{12}$ has linearly independent maximal minors.
Then $\F$ gives a point in $\M(6, 0)$. If the maximal minors of $\psi_{12}$ have no common factor,
then $\F$ is stable. If they have a common linear factor $\ell$, then $\O_L(-1) \subset \F$
is the unique proper subsheaf of slope zero, where $L \subset \P^2$ is the line with equation $\ell =0$.
\end{claim}

\begin{proof}
When the maximal minors of $\psi_{12}$ have no common factor
the claim follows from \ref{2.3.1}
Assume that the maximal minors of $\psi_{12}$ have a common linear factor $\ell$.
We have an extension
\[
0 \lra \O_L(-1) \lra \F \lra \O_C(1) \lra 0,
\]
where $L$ is the line with equation $\ell=0$ and $C$ is a quintic curve.
Thus $\F$ is semi-stable and $\O_L(-1)$, $\O_C(1)$ are its stable factors.
The latter cannot be a subsheaf of $\F$ because $\H^0(\F(-1))$ vanishes.
\end{proof}

\begin{prop}
\label{3.1.5}
The sheaves $\F$ giving points in $\M(6,1)$ and satisfying the conditions
$\h^0(\F(-1))=0$, $\h^1(\F)=2$, $\h^1(\F(1))=0$ are precisely the sheaves having a resolution
\[
0 \lra 2\O(-3) \oplus 2\O(-1) \stackrel{\f}{\lra} \O(-2) \oplus 3\O \lra \F \lra 0
\]
in which $\f_{11}$ has linearly independent entries and $\f_{22}$ has linearly independent maximal minors.
\end{prop}

\begin{proof}
Let $\F$ give a point in $\M(6,1)$ and satisfy the above cohomological conditions.
Display diagram (2.1.1) for the Beilinson spectral sequence I converging to $\F(1)$ reads
\[
\xymatrix
{
5\O(-1) \ar[r]^-{\f_1} & 2\Om^1(1) & 0 \\
0 & 3\Om^1(1) \ar[r]^-{\f_4} & 7\O
}.
\]
Resolving $2\Om^1(1)$ yields the exact sequence
\[
0 \lra \Ker(\f_1) \lra 2\O(-2) \oplus 5\O(-1) \stackrel{\sigma}{\lra} 6\O(-1) \lra \Coker(\f_1) \lra 0.
\]
Arguing as in the proof of \ref{3.1.3}, we see that $\rank(\sigma_{12})=5$,
$\Ker(\f_1) \isom \O(-3)$ and $\Coker(\f_1) \isom \C_x$.
From (2.1.2) we get the exact sequence
\[
0 \lra \O(-3) \oplus 3\Om^1(1) \lra 7\O \lra \Coker(\f_5) \lra 0,
\]
hence the resolution
\[
0 \lra \O(-3) \oplus 9\O \lra 7\O \oplus 3\O(1) \lra \Coker(\f_5) \lra 0.
\]
From (2.1.3) we get the extension
\[
0 \lra \Coker(\f_5) \lra \F(1) \lra \C_x \lra 0.
\]
We apply the horseshoe lemma to the above extension, to the above resolution of $\Coker(\f_5)$
and to the standard resolution of $\C_x$ tensored with $\O(-1)$.
We obtain the exact sequence
\[
0 \lra \O(-3) \lra \O(-3) \oplus 2\O(-2) \oplus 9\O \lra \O(-1) \oplus 7\O \oplus 3\O(1) \lra \F(1) \lra 0.
\]
The map $\O(-3) \to \O(-3)$ is non-zero because $\h^1(\F(1))=0$.
Canceling $\O(-3)$ and tensoring with $\O(-1)$ yields the resolution
\[
0 \lra 2\O(-3) \oplus 9\O(-1) \stackrel{\rho}{\lra} \O(-2) \oplus 7\O(-1) \oplus 3\O \lra \F \lra 0.
\]
Notice that $\rank(\rho_{22})=7$, otherwise $\F$ would map surjectively to the cokernel of a morphism
$2\O(-3) \to \O(-2) \oplus \O(-1)$, in violation of semi-stability. Canceling $7\O(-1)$ we arrive
at a resolution as in the proposition.

\medskip

\noi
Conversely, we assume that $\F$ has a resolution as in the proposition and we need to show that
there are no destabilising subsheaves. From the snake lemma we get an extension
\[
0 \lra \F' \lra \F \lra \C_x \lra 0,
\]
where $\F'$ has a resolution
\[
0 \lra \O(-4) \oplus 2\O(-1) \stackrel{\psi}{\lra} 3\O \lra \F' \lra 0
\]
in which $\psi_{12}=\f_{22}$.
According to \ref{3.1.4}, $\F'$ is semi-stable and the only possible
subsheaf of $\F'$ of slope zero must be of the form $\O_L(-1)$.
It follows that for every subsheaf $\E \subset \F$ we have $\pp(\E) \le 0$ excepting, possibly,
subsheaves that fit into an extension of the form
\[
0 \lra \O_L(-1) \lra \E \lra \C_x \lra 0.
\]
In this case $\E \isom \O_L$ because $\E$ has no zero-dimensional torsion
and we have a diagram similar to diagram (8), leading to a contradiction.
\end{proof}

\begin{prop}
\label{3.1.6}
The sheaves $\F$ giving points in $\M(6,1)$ and satisfying the conditions
$\h^0(\F(-1)) =1$, $\h^1(\F) = 2$
are precisely the sheaves having a resolution of the form
\[
\tag{i}
0 \lra 2\O(-3) \stackrel{\f}{\lra} \O(-1) \oplus \O(1) \lra \F \lra 0,
\]
\[
\f = \left[
\ba{cc}
q_1 & q_2 \\
g_1 & g_2
\ea
\right],
\]
where $q_1$, $q_2$ have no common factor, or the sheaves having a resolution of the form
\[
\tag{ii}
0 \lra 2\O(-3) \oplus \O(-2) \stackrel{\f}{\lra} \O(-2) \oplus \O(-1) \oplus \O(1) \lra \F \lra 0,
\]
\[
\f= \left[
\ba{ccc}
\ell_1 & \ell_2 & 0 \\
q_1 & q_2 & \ell \\
g_1 & g_2 & h
\ea
\right], \qquad \text{where} \qquad \f \nsim \left[
\ba{ccc}
\star & \star & 0 \\
0 & 0 & \star \\
\star & \star & \star
\ea
\right],
\]
$\ell_1$, $\ell_2$ are linearly independent one-forms and $\ell \neq 0$.
\end{prop}

\begin{proof}
Let $\F$ give a point in $\M(6,1)$ and satisfy the above cohomological conditions.
Denote $m = \h^0(\F \tensor \Om^1(1))$.
The Beilinson tableau (2.1.4) for the sheaf $\G=\F^\D(1)$ reads
\[
\xymatrix
{
3\O(-2) \ar[r]^-{\f_1} & m\O(-1) \ar[r]^-{\f_2} & \O \\
2\O(-2) \ar[r]^-{\f_3} & (m+4)\O(-1) \ar[r]^-{\f_4} & 6\O
}.
\]
Since $\f_2$ is surjective, $m \ge 3$.
Since $\G$ maps surjectively to $\CC = \Ker(\f_2)/\Im(\f_1)$, $m \le 4$.
If $m=4$, then $\pp(\CC)= -1/2$, violating the semi-stability of $\G$. Thus $m=3$.
As at 2.2.4 \cite{mult_five}, we have $\Ker(\f_2) = \Im(\f_1)$ and $\Ker(\f_1) \isom \O(-3)$.
As at 3.2.5 \cite{mult_five}, it can be shown that $\Coker(\f_3) \isom 2\Om^1(1) \oplus \O(-1)$.
Combining the exact sequences (2.1.5) and (2.1.6) we obtain the resolution
\[
0 \lra \O(-3) \oplus 2\Om^1(1) \oplus \O(-1) \lra 6\O \lra \G \lra 0.
\]
Dualising and resolving $2\Om^1$ leads to the resolution
\[
0 \lra 2\O(-3) \oplus 6\O(-2) \stackrel{\rho}{\lra} 6\O(-2) \oplus \O(-1) \oplus \O(1) \lra \F \lra 0.
\]
Note that $\rank(\rho_{12}) \ge 5$, otherwise $\F$ would map surjectively to the cokernel
of a morphism $2\O(-3) \to 2\O(-2)$, in violation of semi-stability.
When $\rank(\rho)=5$ we get resolution (ii).
When $\rank(\rho)=6$ we get resolution (i).

Conversely, if $\F$ has resolution (i), then, in view of \ref{2.3.2}, $\F$ is stable.
Assume now that $\F$ has resolution (ii).
We examine first the case when $\ell$ does not divide $h$.
From the snake lemma we have an extension
\[
0 \lra \F' \lra \F \lra \C_x \lra 0,
\]
where $\F'$ is the cokernel of a morphism $\psi \colon \O(-4) \oplus \O(-2) \to \O(-1) \oplus \O(1)$
for which $\psi_{12}$ does not divide $\psi_{22}$.
In view of \ref{2.3.2}, $\F'$ is semi-stable and the only possible subsheaf of $\F'$
of slope zero must be of the form $\O_C(1)$, for a quintic curve $C \subset \P^2$.
It follows that every proper subsheaf of $\F$ has non-positive slope except, possibly,
extensions $\E$ of $\C_x$ by $\O_C(1)$.
According to 3.1.5 \cite{mult_five}, we have a resolution
\[
0 \lra 2\O(-3) \lra \O(-2) \oplus \O(1) \lra \E \lra 0.
\]
This forms part of a diagram analogous to diagram (8),
leading to a contradiction.

Assume now that $\ell$ divides $h$. We may assume that $h=0$.
Let $L$ be the line given by the equation $\ell=0$.
From the snake lemma we get a non-split extension
\[
0 \lra \O_L(-1) \lra \F \lra \E \lra 0,
\]
where $\E$ is as above.
According to loc.cit., $\E$ is stable. It is easy to see now that $\F$ is stable as well
\end{proof}

\begin{prop}
\label{3.1.7}
\emph{(i)} The sheaves $\G$ giving points in $\M(6,4)$ and satisfying the condition $\h^0(\G(-2)) > 0$
are precisely the sheaves having a resolution of the form
\[
0 \lra 2\O(-3) \stackrel{\f}{\lra} \O(-2) \oplus \O(2) \lra \G \lra 0,
\]
\[
\f = \left[
\ba{cc}
\ell_1 & \ell_2 \\
f_1 & f_2
\ea
\right],
\]
where $\ell_1, \ell_2$ are linearly independent one-forms.

\medskip

\noi
\emph{(ii)} By duality, the sheaves $\F$ giving points in $\M(6,2)$ and satisfying the condition
$\h^1(F(1)) > 0$ are precisely the sheaves having resolution
\[
0 \lra \O(-4) \oplus \O \stackrel{\f^\T}{\lra} 2\O(1) \lra \F \lra 0.
\]
These are precisely the sheaves of the form $\J_x(2)$, where $\J_x \subset \O_C$
is the ideal sheaf of a closed point $x$ inside a sextic curve $C \subset \P^2$.
\end{prop}

\begin{proof}
The argument is entirely analogous to the argument at 3.1.5 \cite{mult_five}.
\end{proof}

\begin{prop}
\label{3.1.8}
The sheaves $\F$ giving points in $\M(6,1)$ and satisfying the condition $\h^1(\F(1)) > 0$
are precisely the sheaves having a resolution of the form
\[
0 \lra \O(-4) \oplus \O(-1) \stackrel{\f}{\lra} \O \oplus \O(1) \lra \F \lra 0,
\]
\[
\f= \left[
\ba{cc}
h & \ell \\
g & q
\ea
\right],
\]
where $\ell \neq 0$ and $\ell$ does not divide $q$.
These are precisely the sheaves of the form $\J_Z(2)$, where $\J_Z \subset \O_C$
is the ideal sheaf of a zero dimensional subscheme $Z$ of length $2$ inside a sextic
curve $C \subset \P^2$.
\end{prop}

\begin{proof}
Let $\F$ give a point in $\M(6,1)$ and satisfy the condition $\h^1(\F(1)) > 0$.
Denote $\G = \F^\D(1)$. According to \cite{maican-duality}, $\G$ gives a point in $\M(6,5)$
and $\h^0(\G(-2)) > 0$.
As in the proof of 2.1.3 \cite{drezet-maican}, there is an injective morphism $\O_C \to \G(-2)$,
where $C \subset \P^2$ is a curve.
Clearly $C$ has degree $6$, otherwise $\O_C$ would destabilise $\G(-2)$.
The quotient sheaf $\CC= \G/\O_C(2)$ has support of dimension zero and length $2$.
Write $\CC$ as an extension of $\O_{\P^2}$-modules of the form
\[
0 \lra \C_x \lra \CC \lra \C_y \lra 0.
\]
Let $\G'$ be the preimage of $\C_x$ in $\G$.
This subsheaf has no zero-dimensional torsion and is an extension of $\C_x$ by $\O_C(2)$
hence, in view of \ref{3.1.7}, it has a resolution of the form
\[
0 \lra 2\O(-3) \lra \O(-2) \oplus \O(2) \lra \G' \lra 0.
\]
We construct a resolution of $\G$ from the above resolution
of $\G'$ and from the standard resolution of $\C_y$ tensored with $\O(-1)$:
\[
0 \lra \O(-3) \lra 2\O(-3) \oplus 2\O(-2) \lra \O(-2) \oplus \O(-1) \oplus \O(2) \lra \G \lra 0.
\]
If the morphism $\O(-3) \to 2\O(-3)$ were zero, then it could be shown, as in the proof
of 2.3.2 \cite{mult_five}, that $\C_y$ is a direct summand of $\G$.
This would contradict our hypothesis. Thus we may cancel $\O(-3)$ to get the resolution
\[
0 \lra \O(-3) \oplus 2\O(-2) \lra \O(-2) \oplus \O(-1) \oplus \O(2) \lra \G \lra 0.
\]
If the morphism $2\O(-2) \to \O(-2)$ were zero, then $\G$ would have a destabilising quotient
sheaf of the form $\O_L(-2)$.
Thus we may cancel $\O(-2)$ to get a resolution
\[
0 \lra \O(-3) \oplus \O(-2) \stackrel{\psi}{\lra} \O(-1) \oplus \O(2) \lra \G \lra 0,
\]
\[
\psi = \left[
\ba{cc}
q & \ell \\
g & h
\ea
\right],
\]
in which $\ell \neq 0$ and $\ell$ does not divide $q$.
Dualising, we get a resolution for $\F$ as in the proposition.
The converse follows from \ref{2.3.2}.
\end{proof}

\noi
In the remaining part of this subsection we shall prove that there are no
sheaves $\F$ giving points in $\M(6,1)$ beside the sheaves we have discussed so far.
In view of \ref{3.1.8} we may restrict our attention to the case when $\H^1(\F(1))=0$.
Assume that $\h^0(\F(-1)) \le 1$. According to \ref{2.2.2}(i), (ii), and \ref{3.1.1}
the pair $(\h^0(\F(-1)), \h^1(\F))$ may be one of the following: $(0, 0)$, $(0, 1)$, $(0, 2)$, $(1, 2)$.
Each of these situations has already been examined.
The following concludes the classification of sheaves in $\M(6,1)$:

\begin{prop}
\label{3.1.9}
Let $\F$ be a sheaf giving a point in $\M(6,1)$. Then $\h^0(\F(-1))=0$ or $1$.
\end{prop}

\begin{proof}
Assume that $\F$ gives a point in $\M(6,1)$ and $\h^0(\F(-1))>0$.
As in the proof of 2.1.3 \cite{drezet-maican}, there is an injective morphism $\O_C \to \F(-1)$
for a curve $C \subset \P^2$. From the semi-stability of $\F$ we see that $C$ has degree $5$ or $6$.
In the first case $\F(-1)/\O_C$ has Hilbert polynomial $\PP(m)=m$ and has no zero-dimensional torsion.
Indeed, the pull-back in $\F(-1)$ of any non-zero subsheaf of $\F(-1)/\O_C$ supported on finitely
many points would destabilise $\F(-1)$. We deduce that $\F(-1)/\O_C$ is isomorphic to
$\O_L(-1)$, hence $\h^0(\F(-1))=1$.

Assume now that $C$ is a sextic curve and $\H^1(\F(1))=0$.
The quotient sheaf $\CC= \F(-1)/\O_C$ has support of dimension zero and length $4$.
Assume that $\h^0(\F(-1))> 1$. Then, in view of \ref{2.2.2}(iii), we have $\h^0(\F(-1)) \ge 3$.
We claim that there is a global section $s$ of $\F(-1)$ such that its image in $\CC$
generates a subsheaf isomorphic to $\O_Z$, where $Z \subset \P^2$ is a zero-dimensional
scheme of length $1$, $2$ or $3$.
Indeed, as $\h^0(\O_C)=1$ and $\h^0(\F(-1)) \ge 3$
there are global sections $s_1$ and $s_2$ of $\F(-1)$ such that their images in $\CC$
are linearly independent.
Consider a subsheaf $\CC' \subset \CC$ of length $3$.
Choose $c_1, c_2 \in \C$, not both zero, such that the image of $c_1 s_1 + c_2 s_2$
under the composite map $\F(-1) \to \CC \to \CC/\CC'$ is zero.
Then $s=c_1 s_1 + c_2 s_2$ satisfies our requirements.

Let $\F' \subset \F(-1)$ be the preimage of $\O_Z$. Assume first that $Z$ is not contained in a line,
so, in particular, it has length $3$. According to \cite{modules-alternatives}, proposition 4.5, we have
a resolution
\[
0 \lra 2\O(-3) \lra 3\O(-2) \lra \O \lra \O_Z \lra 0.
\]
Combining this with the standard resolution of $\O_C$ we obtain the exact sequence
\[
0 \lra 2\O(-3) \lra \O(-6) \oplus 3\O(-2) \lra 2\O \lra \F' \lra 0.
\]
As the morphism $2\O(-3) \to \O(-6)$ in the above complex is zero and as $\Ext^1(\O_Z, \O)$
vanishes, we can show, as in the proof of 2.3.2 \cite{mult_five},
that $\O_Z$ is a direct summand of $\F'$.
This is absurd, by hypothesis $\F(-1)$ has no zero-dimensional torsion.
The same argument applies if $Z$ is contained in a line and has length $3$,
except that this time we use the resolution
\[
0 \lra \O(-4) \lra \O(-3) \oplus \O(-1) \lra \O \lra \O_Z \lra 0.
\]
The cases when $\length(Z)= 1$ or $2$ are analogous.
Thus $\h^0(\F(-1))=1$.
\end{proof}

\subsection{The strata as quotients} 

In the previous subsection we classified all sheaves giving points in $\M(6,1)$, namely we showed
that this moduli space can be decomposed into six subsets $X_0, \ldots, X_5$, cf. Table 1.
Recall the notations $\W_i$, $W_i$, $G_i$, $0 \le i \le 5$, from subsection 1.2.
The fibres of the canonical maps $\rho_i \colon W_i \to X_i$ are precisely the $G_i$-orbits.
Given $[\F] \in X_i$, we constructed $\f \in \rho_i^{-1}[\F]$ starting from the Beilinson spectral sequence
I or II associated to $\F$ or some twist of this sheaf and performing algebraic operations.
This construction is local in the sense that it can be done for flat families of sheaves that are
in a sufficiently small neighbourhood of $[\F]$. This allows us to deduce, as at 3.1.6 \cite{drezet-maican},
that the maps $\rho_i$ are categorical quotient maps.
Applying \cite{mumford}, remark 2, page 5, it follows that $X_i$ is normal.
From \cite{popov-vinberg}, theorem 4.2, we conclude that each $\rho_i$ is a geometric quotient map.

Some of these quotients have concrete descriptions.
The quotient $W_5/G_5$ is isomorphic to the flag Hilbert scheme of pairs $(C,Z)$, where $C \subset \P^2$
is a curve of degree $6$ and $Z \subset C$ is a zero-dimensional scheme of length $2$.
Let $W_0' \subset \W_0$ be the set of morphisms $\f$ for which $\f_{11}$ is semi-stable
as a Kronecker module and $\f_{21} \neq v \f_{11}$ for any $v \in \Hom(4\O(-1), \O)$.
Clearly $W_0 \subsetneqq W_0'$, being the subset of injective morphisms.
According to 9.3 \cite{drezet-trautmann}, the geometric quotient
$W_0'/G_0$ exists and is the projectivisation of a certain vector bundle over $\N(3,5,4)$ of rank $18$.
Clearly $W_0/G_0$ is a proper open subset of $W_0'/G_0$.

The quotient $W_3/G_3$ can be constructed as at 2.2.2 \cite{mult_five}.
Let $W_3' \subset \W_3$ be the subset given by the following conditions:
$\f_{12}=0$, $\f_{11}$ has linearly independent entries, $\f_{22}$ has linearly independent maximal minors,
$\f_{21} \neq \f_{22} u + v \f_{11}$ for any $u \in \Hom(2\O(-3), 2\O(-1))$ and $v \in \Hom(\O(-2), 3\O)$.
Clearly $W_3 \subsetneqq W_3'$, being the subset of injective morphisms.
Let $U_3$ be the set of pairs $(\f_{11},\f_{22})$ satisfying the above properties
and let $\Gamma_3$ be the canonical group acting on $U_3$.
Applying the method of loc.cit. one can show that the quotient $W_3'/G_3$ exists and is the projectivisation
of a vector bundle of rank $24$ over $U_3/\Gamma_3 \isom \P^2 \times \N(3,2,3)$.
Thus $W_3/G_3$ is a proper open subset of $W_3'/G_3$.
Analogously one can construct the quotient $W_2/G_2$ except that this time one has to pay special
attention to the fact that the canonical group acting on the space of triples $(\f_{11}, \f_{12}, \f_{23})$
satisfying the properties of \ref{3.1.3}(ii) is non-reductive.

\begin{claim}
\label{3.2.1}
Let $\U = \Hom(\O(-3) \oplus 2\O(-2), 2\O(-1))$ and let $U \subset \U$ be the set of morphisms
\[
\psi = \left[
\ba{ccc}
q_1 & \ell_{11} & \ell_{12} \\
q_2 & \ell_{21} & \ell_{22}
\ea
\right]
\]
that satisfy the conditions of \ref{3.1.3}(ii).
Let $G$ be the canonical group acting by conjugation on $U$.
Then there exists a geometric quotient $U/G$,
which is a smooth projective variety of dimension $10$.
\end{claim}

\begin{proof}
It is straightforward to check that the conditions defining $U$ are equivalent to saying that
$\psi$ be not equivalent to a morphism represented by a matrix having one of the following forms:
\[
\left[
\ba{ccc}
\star & 0 & 0 \\
\star & \star & \star
\ea
\right], \qquad \left[
\ba{ccc}
\star & \star & 0 \\
\star & \star & 0
\ea
\right], \qquad \left[
\ba{ccc}
0 & 0 & \star \\
\star & \star & \star
\ea
\right], \qquad \left[
\ba{ccc}
0 & \star & \star \\
0 & \star & \star
\ea
\right].
\]
This allows us to interpret $U$ as the set of semi-stable points in the sense of \cite{drezet-trautmann}.
Adopting the notations of op.cit., let $\L=(\l_1, \l_2, \m_1)$ be a polarisation for the action of $G$ on $\U$
satisfying the condition $1/4 < \l_2 < 1/2$. Using King's criterion of semi-stability \cite{king}
and the above alternate description of $U$ we deduce that $U$ is the set $\U^\st(\L)$ of morphisms
that are stable relative to $\L$ (cf. \cite{drezet-trautmann}). According to op.cit., propositions 6.1.1, 7.2.2,
8.1.3, there exists a geometric quotient $\U^\st(\L)/G$, which is a smooth quasi-projective variety,
provided $3/7 < \l_2 < 1/2$, which we assume to be the case.
This quotient is projective because $\U^\st(\L)$ coincides with the set of semi-stable points
in $\U$ relative to $\L$.
\end{proof}

\noi
The quotient $W_2/G_2$ is an open subset of the projectivisation of a vector bundle
over $(U/G) \times \P^2$ of rank $22$.

\subsection{Generic sheaves}  

Let $C \subset \P^2$ denote an arbitrary smooth sextic curve and let $P_i$ denote distinct
points on $C$. According to \cite{modules-alternatives}, propositions 4.5 and 4.6,
the cokernels of morphisms $4\O(-5) \to 5 \O(-4)$ whose maximal minors have no common factor
are precisely the ideal sheaves $\I_Z \subset \O_{\P^2}$
of zero-dimensional schemes $Z \subset \P^2$ of length $10$
that are not contained in a cubic curve.
It follows that the generic sheaves giving points in $X_0$
are of the form $\O_C(P_1 + \cdots + P_{10})$,
where $P_i$, $1 \le i \le 10$, are not contained in a cubic curve.

According to loc.cit., the cokernels of morphisms $2\O(-3) \to 3\O(-2)$
whose maximal minors have no common factor
are precisely the ideal sheaves $\I_Z \subset \O_{\P^2}$
of zero-dimensional schemes $Z \subset \P^2$ of length $3$
that are not contained in a line.
It follows that the generic sheaves in $X_3$ have the form $\O_C(2)(-P_1 - P_2 - P_3 + P_4)$,
where $P_1, P_2, P_3$ are non-colinear.

Obviously, the generic sheaves in $X_4$ have the form $\O_C(1)(P_1+P_2+P_3+P_4)$,
where no three points among $P_1, P_2, P_3, P_4$ are colinear.
Also, the generic sheaves in $X_5$ are of the form $\O_C(2)(-P_1-P_2)$.
According to claim \ref{3.3.1} below,
the generic sheaves in $X_1$ have the form $\O_C(3)(-P_1 - \cdots - P_8)$,
where no four points among $P_1, \ldots, P_8$ are colinear
and no seven of them lie on a conic curve.
According to claim \ref{3.3.2} below,
the generic sheaves in $X_2$ have the form $\O_C(1)(P_1+ \cdots + P_5 - P_6)$,
where no three points among $P_1, \ldots, P_5$ are colinear.

\begin{claim}
\label{3.3.1}
Let $U \subset \Hom(2\O(-2), \O(-1) \oplus 2\O)$ be the set of morphisms
represented by matrices
\[
\left[
\ba{cc}
\ell_1 & \ell_2 \\
q_{11} & q_{12} \\
q_{21} & q_{22}
\ea
\right]
\]
for which the maximal minors
$\ell_1 q_{12} - \ell_2 q_{11}$ and $\ell_1 q_{22} - \ell_2 q_{21}$ have no common factor.
The cokernels of the morphisms in $U$ are precisely the sheaves of the form $\I_Z(3)$,
where $\I_Z \subset \O_{\P^2}$ is the ideal sheaf of a zero-dimensional subscheme $Z \subset \P^2$
of length $8$, no subscheme of length $4$ of which is contained in a line and no subscheme
of length $7$ of which is contained in a conic curve.
\end{claim}

\begin{proof}
Let $\psi \in U$ and let $\z_i$ denote the maximal minor of $\psi$ obtained by deleting the $i$-th row.
Since $\z_1, \z_2, \z_3$ have no common factor,
there is an exact sequence of the form
\[
0 \lra 2\O(-2) \stackrel{\psi}{\lra} \O(-1) \oplus 2\O \stackrel{\z}{\lra} \O(3) \lra \CC \lra 0,
\]
\[
\z = \left[
\ba{ccc}
\z_1 & - \z_2 & \phantom{-} \z_3
\ea
\right].
\]
The Hilbert polynomial of $\CC$ is $8$, hence $\CC$ is the structure sheaf of a zero-dimensional
scheme $Z$ of length $8$ and $\Coker(\psi) \isom \I_Z(3)$.
If four of the points of $Z$ were on the line with equation $\ell = 0$,
then, by B\'ezout's theorem, $\ell$ would divide $\z_2$ and $\z_3$, contrary to our hypothesis.
Similarly, if seven of the points of $Z$ lay on the irreducible conic curve with equation $q = 0$,
then $q$ would divide $\z_2$ and $\z_3$.

For the converse we use the method of 4.5 \cite{modules-alternatives}.
Assume that $Z \subset \P^2$ is a subscheme as in the proposition.
The Beilinson spectral sequence I with $\EE^1$-term
\[
\EE^1_{ij} = \H^j (\I_Z(2) \tensor \Om^{-i}(-i)) \tensor \O(i)
\]
converges to $\I_Z(2)$. By hypothesis $\H^0(\I_Z(2))= 0$, hence also
$\H^0(\I_Z(3) \tensor \Om^1)=0$ and $\H^0(\I_Z(1)) = 0$.
Using Serre duality we can show that $\H^2(\I_Z(2))$, $\H^2(\I_Z(1))$ and $\H^2(\I_Z(3) \tensor \Om^1)$
vanish. The middle row of the display diagram for the Beilinson
spectral sequence yields a monad
\[
0 \lra 5\O(-2) \stackrel{\a}{\lra} 8\O(-1) \stackrel{\b}{\lra} 2\O \lra 0
\]
with middle cohomology $\I_Z(2)$. Denote $\B = {\mathcal Hom}(\Ker(\b), \O(-1))$.
Applying the functor ${\mathcal Hom}(\_\_, \O(-1))$ we get the exact sequences
\[
0 \lra 2\O(-1) \stackrel{\b^\T}{\lra} 8\O \lra \B \lra 0,
\]
\[
0 \lra {\mathcal Hom}(\I_Z(2), \O(-1)) \lra \B \lra 5\O(1).
\]
From the first exact sequence we see that $\h^0(\B)=8$ and from the second exact
sequence we see that $\B$ is torsion-free.
It follows that the morphism $8\O \to \B$ cannot factor through $7\O \oplus \C_x$.
This allows us to deduce, as at 2.1.4 \cite{mult_five}, that any matrix representing $\b^\T$
has at least three linearly independent entries on each column,
in other words, that $\b^\T$ has one of the following canonical forms:
\[
\left[
\ba{cc}
0 & 0 \\
0 & 0 \\
0 & 0 \\
0 & 0 \\
0 & 0 \\
X & R \\
Y & S \\
Z & T
\ea
\right], \qquad \left[
\ba{cc}
0 & 0 \\
0 & 0 \\
0 & 0 \\
0 & 0 \\
X & 0 \\
Y & R \\
Z & S \\
0 & T
\ea
\right], \qquad \left[
\ba{cc}
0 & 0 \\
0 & 0 \\
0 & 0 \\
X & 0 \\
Y & 0 \\
Z & R \\
0 & S \\
0 & T
\ea
\right], \qquad \left[
\ba{cc}
0 & 0 \\
0 & 0 \\
X & 0 \\
Y & 0 \\
Z & 0 \\
0 & X \\
0 & Y \\
0 & Z
\ea
\right].
\]
Moreover, the morphism $8\O \to \B$ cannot factor through $6\O \oplus \O_L(1)$.
This allows us to deduce, as at 3.1.3 \cite{mult_five}, that the first three canonical forms
are unfeasible.
Thus $\Ker(\b) \isom 2\Om^1 \oplus 2\O(-1)$, so we have a resolution
\[
0 \lra 5\O(-2) \lra 2\Om^1 \oplus 2\O(-1) \lra \I_Z(2) \lra 0,
\]
hence a resolution
\[
0 \lra 2\O(-3) \oplus 5\O(-2) \stackrel{\rho}{\lra} 6\O(-2) \oplus 2\O(-1) \lra \I_Z(2) \lra 0.
\]
Notice that $\rank(\rho_{12}) \ge 3$, otherwise $\I_Z(2)$ would map surjectively
onto the cokernel of a morphism $2\O(-3) \to 4\O(-2)$, which is impossible,
because $\rank(\I_Z(2))=1$.
Assume that $\rank(\rho_{12}) = 3$. We get a resolution
\[
0 \lra 2\O(-3) \oplus 2\O(-2) \stackrel{\eta}{\lra} 3\O(-2) \oplus 2\O(-1) \lra \I_Z(2) \lra 0
\]
with $\eta_{12}=0$. Clearly $\eta_{22}$ is injective and $\Coker(\eta_{22})$ maps injectively
to $\I_Z(2)$. This is absurd, $\I_Z(2)$ is a torsion-free sheaf whereas $\Coker(\eta_{22})$
is a torsion sheaf.
Assume that $\rank(\rho_{12})= 4$.
We have a resolution
\[
0 \lra 2\O(-3) \oplus \O(-2) \stackrel{\eta}{\lra} 2\O(-2) \oplus 2\O(-1) \lra \I_Z(2) \lra 0
\]
with $\eta_{12}=0$. The entries of $\eta_{22}$ are linearly independent,
otherwise $\I_Z(2)$ would have a subsheaf of the form $\O_L(-1)$, which is absurd.
Let $x$ be the common zero of the entries of $\eta_{22}$.
The points of $Z$ distinct from $x$ lie on the conic curve with equation $\det(\eta_{11})=0$,
contradicting our hypothesis on $Z$.
We conclude that $\rank(\rho_{12}) = 5$ and we arrive at the resolution
\[
0 \lra 2\O(-3) \stackrel{\psi}{\lra} \O(-2) \oplus 2\O(-1) \lra \I_Z(2) \lra 0,
\]
\[
\psi = \left[
\ba{cc}
\ell_1 & \ell_2 \\
q_{11} & q_{12} \\
q_{21} & q_{22}
\ea
\right].
\]
We will show that $\psi$ satisfies the conditions defining $U$.
Assume that $\gcd(\z_2, \z_3)$ is a linear form $\ell$.
By hypothesis, at least five points of $Z$ do not lie on the line given by the equation
$\ell =0$. These points must be then in the common zero-set of $\z_2/\ell$ and $\z_3/\ell$,
which, by B\'ezout's theorem, is impossible.
Likewise, $\gcd(\z_2, \z_3)$ cannot be a quadratic form.
If $\z_2$ divided $\z_3$, then, performing possibly row operations on $\psi$,
we may assume that $\z_3=0$.
It would follow that
\[
\psi \sim \left[
\ba{cc}
\star & \star \\
0 & 0 \\
\star & \star
\ea
\right] \qquad \text{or} \qquad \psi \sim \left[
\ba{cc}
\star & 0 \\
\star & 0 \\
\star & \star
\ea
\right].
\]
In each case $\I_Z(2)$ would have a torsion subsheaf, which is absurd.
\end{proof}

\begin{claim}
\label{3.3.2}
Let $U \subset \Hom(2\O(-1), 2\O \oplus \O(1))$ be the set of morphisms
represented by matrices
\[
\left[
\ba{cc}
\ell_{11} & \ell_{12} \\
\ell_{21} & \ell_{22} \\
q_1 & q_2
\ea
\right]
\]
such that $\z_3 =\ell_{11} \ell_{22} - \ell_{12} \ell_{21}$ is irreducible and does not divide any
of the other maximal minors. The cokernels of the morphisms in $U$
are precisely the sheaves of the form $\I_Z(3)$,
where $\I_Z \subset \O_{\P^2}$ is the ideal sheaf of a zero-dimensional subscheme $Z \subset \P^2$
of length $5$, no subscheme of length $3$ of which is contained in a line.
\end{claim}

\begin{proof}
As in the proof of \ref{3.3.1}, one direction is obvious.
Assume now that $Z \subset \P^2$ is a scheme as above.
By hypothesis $\H^0(\I_Z(1))$ vanishes, hence also
$\H^0 (\I_Z(2) \tensor \Om^1)$ and $\H^0(\I_Z)$ vanish.
By Serre duality, $\H^2(\I_Z(1))$, $\H^2(\I_Z(2) \tensor \Om^1)$ and $\H^2(\I_Z)$ vanish.
The middle row of the display diagram for the Beilinson spectral sequence I converging
to $\I_Z(1)$ yields a monad
\[
0 \lra 4\O(-2) \stackrel{\a}{\lra} 7\O(-1) \stackrel{\b}{\lra} 2\O \lra 0
\]
with middle cohomology $\I_Z(1)$.
As at \ref{3.3.1}, $\Ker(\b) \isom 2\Om^1 \oplus \O(-1)$,
so we have a resolution
\[
0 \lra 4\O(-2) \lra 2\Om^1 \oplus \O(-1) \lra \I_Z(1) \lra 0,
\]
leading to a resolution
\[
0 \lra 2\O(-3) \oplus 4\O(-2) \stackrel{\rho}{\lra} 6\O(-2) \oplus \O(-1) \lra \I_Z(1) \lra 0.
\]
Note that $\rank(\rho_{12}) \ge 3$, otherwise $\I_Z(1)$ would map surjectively onto the cokernel
of a morphism $2\O(-3) \to 4\O(-2)$, which is impossible, due to the fact that $\rank(\I_Z(1))=1$.
Assume that $\rank(\rho_{12})=3$. We obtain a resolution
\[
0 \lra 2\O(-3) \oplus \O(-2) \stackrel{\eta}{\lra} 3\O(-2) \oplus \O(-1) \lra \I_Z(1) \lra 0
\]
with $\eta_{12} = 0$. Clearly $\Coker(\eta_{22})$ maps injectively to $\I_Z(1)$.
This is absurd, $\I_Z(1)$ is a torsion-free sheaf whereas $\Coker(\eta_{22})$ is a torsion sheaf.
We conclude that $\rank(\rho_{12})=4$ and we arrive at the resolution
\[
0 \lra 2\O(-3) \stackrel{\psi}{\lra} 2\O(-2) \oplus \O(-1) \lra \I_Z(1) \lra 0.
\]
We will show that $\psi$ satisfies the conditions defining $U$.
Clearly, the maximal minors of $\psi$ have no common factor and generate the ideal sheaf of $Z$.
If the conic curve given by the equation $\z_3=0$ were reduced, then at least three points of $Z$
would lie on one of its components.
If $\z_3$ divided $\z_1$, then it would also divide $\z_2$.
\end{proof}


\section{The moduli space $\M(6,2)$}

\subsection{Classification of sheaves}    

\begin{prop}
\label{4.1.1}
Every sheaf $\F$ giving a point in $\M(6,2)$ and satisfying the condition $\h^1(\F)=0$
also satisfies the condition $\h^0(\F(-1))=0$.
For these sheaves $\h^0(\F \tensor \Om^1(1))= 0$ or $1$.
The sheaves from the first case are precisely the sheaves having a resolution of the form
\[
\tag{i}
0 \lra 4\O(-2) \stackrel{\f}{\lra} 2\O(-1) \oplus 2\O \lra \F \lra 0,
\]
where $\f$ is not equivalent, modulo the action of the natural group of automorphisms,
to a morphism represented by a matrix having one of the following forms:
\[
\left[
\ba{cccc}
\star & 0 & 0 & 0 \\
\star & \star & \star & \star \\
\star & \star & \star & \star \\
\star & \star & \star & \star
\ea
\right], \qquad \left[
\ba{cccc}
\star & \star & 0 & 0 \\
\star & \star & 0 & 0 \\
\star & \star & \star & \star \\
\star & \star & \star & \star
\ea
\right], \qquad \left[
\ba{cccc}
\star & \star & \star & 0 \\
\star & \star & \star & 0 \\
\star & \star & \star & 0 \\
\star & \star & \star & \star
\ea
\right].
\]
The sheaves in the second case are precisely the sheaves with resolution of the form
\[
\tag{ii}
0 \lra 4\O(-2) \oplus \O(-1) \stackrel{\f}{\lra} 3\O(-1) \oplus 2\O \lra \F \lra 0,
\]
where $\f_{12}=0$, $\f_{11}$ is semi-stable as a Kronecker module and
$\f_{22}$ has linearly independent entries.
\end{prop}

\begin{proof}
The first statement follows from 6.4 \cite{maican}.
The rest of the proposition follows by duality from 4.3 op.cit.
\end{proof}

\begin{prop}
\label{4.1.2}
The sheaves $\F$ giving points in $\M(6,2)$ and satisfying the conditions
$\h^0(\F(-1))=0$, $\h^1(\F)=1$, $\h^1(\F(1))=0$ are precisely
the sheaves having a resolution of the form
\[
\tag{i}
0 \lra \O(-3) \oplus \O(-2) \oplus \O(-1) \stackrel{\f}{\lra} 3\O \lra \F \lra 0,
\]
where $\f$ is not equivalent to a morphism of any of the following forms
\[
\f_1 = \left[
\ba{ccc}
\star & \star & \star \\
\star & \star & 0 \\
\star & \star & 0
\ea
\right], \qquad \f_2 = \left[
\ba{ccc}
\star & \star & \star \\
\star & 0 & \star \\
\star & 0 & \star
\ea
\right], \qquad \f_3 = \left[
\ba{ccc}
\star & \star & \star \\
\star & \star & \star \\
\star & 0 & 0
\ea
\right],
\]
or the sheaves having a resolution of the form
\[
\tag{ii}
0 \lra \O(-3) \oplus \O(-2) \oplus 2\O(-1) \stackrel{\f}{\lra} \O(-1) \oplus 3\O \lra \F \lra 0,
\]
where $\f_{12} \neq 0$, $\f_{13} = 0$, $\f_{11}$ is not divisible by $\f_{12}$ and
$\f_{23}$ has linearly independent maximal minors.
\end{prop}

\begin{proof}
Let $\F$ give a point in $\M(6,2)$ and satisfy the above cohomological conditions.
Display diagram (2.1.1) for the Beilinson spectral sequence I converging to $\F(1)$ reads
\[
\xymatrix
{
4\O(-1) \ar[r]^-{\f_1} & \Om^1(1) & 0 \\
0 & 3\Om^1(1) \ar[r]^-{\f_4} & 8\O
}.
\]
As in the proof of \ref{3.1.3}, we have $\Coker(\f_1)=0$, $\Ker(\f_1) \isom \O(-2) \oplus \O(-1)$.
Performing the same steps as at loc.cit. we arrive at the resolution
\[
0 \lra \O(-2) \oplus \O(-1) \oplus 9\O \stackrel{\rho}{\lra} 8\O \oplus 3\O(1) \lra \F(1) \lra 0.
\]
Notice that $\rank(\rho_{13}) \ge 7$, otherwise $\F(1)$ would map surjectively to the cokernel
of a morphism $\O(-2) \oplus \O(-1) \to 2\O$, in violation of semi-stability.
From here on we get resolution (i) or (ii), depending on whether $\rank(\rho_{13})=8$ or $7$.

\medskip

\noi
Conversely, we assume that $\F$ has resolution (i) and we need to show that
there are no destabilising subsheaves. Assume that $\E \subset \F$ is a destabilising subsheaf.
We may take $\E$ to be semi-stable.
As $\F$ is generated by global sections, we have $\h^0(\E) < \h^0(\F)$.
Thus $\E$ gives a point in $\M(r,1)$ or $\M(r,2)$ for some $r$, $1 \le r \le 5$.
The case when $\PP_{\E}(m)=3m+1$ can be easily ruled out.
Moreover, we have $\h^0(\E(-1))=0$, $\h^0(\E \tensor \Om^1(1)) \le 1$.
From the results in \cite{drezet-maican} and \cite{mult_five} we see that $\E$ may have
one of the following resolutions:
\[
\tag{1}
0 \lra \O(-1) \lra \O \lra \E \lra 0,
\]
\[
\tag{2}
0 \lra \O(-2) \lra \O \lra \E \lra 0,
\]
\[
\tag{3}
0 \lra \O(-2) \oplus \O(-1) \lra 2\O \lra \E \lra 0,
\]
\[
\tag{4}
0 \lra 2\O(-2) \lra 2\O \lra \E \lra 0,
\]
\[
\tag{5}
0 \lra 2\O(-2) \oplus \O(-1) \lra \O(-1) \oplus 2\O \lra \E \lra 0,
\]
\[
\tag{6}
0 \lra 3\O(-2) \lra \O(-1) \oplus 2\O \lra \E \lra 0,
\]
\[
\tag{7}
0 \lra 3\O(-2) \oplus \O(-1) \lra 2\O(-1) \oplus 2\O \lra \E \lra 0.
\]
Each of these resolutions must fit into a commutative diagram like diagram (8) at \ref{3.1.3}
in which $\a$ is injective on global sections.
For the first four resolutions $\a$ must be injective and we get the contradictory conclusions
that $\f \sim \f_1$, $\f \sim \f_2$ or $\f \sim \f_3$.
If $\E$ has resolution (5), then $\b$ cannot be injective, hence $\a$ is not injective,
hence $\Ker(\a) \isom \Ker(\b) \isom \O(-1)$ and we conclude, as in the case of resolution (4),
that $\f \sim \f_3$. If $\E$ has resolution (6), then, again, $\Ker(\a) \isom \O(-1) \isom \Ker(\b)$,
which is absurd, because $\O(-1)$ cannot be isomorphic to a subsheaf of $3\O(-2)$.
For resolution (7) we arrive at a contradiction in a similar manner.

\medskip

\noi
Assume now that $\F$ has resolution (ii).
Assume that there is a destabilising subsheaf $\E \subset \F$.
We may assume that $\E$ is semi-stable.
From the snake lemma we obtain an extension
\[
0 \lra \F' \lra \F \lra \O_Z \lra 0,
\]
where $Z$ is the zero-dimensional scheme of length $2$
given by the ideal $(\f_{11}, \f_{12})$ and $\F'$ has a resolution
\[
0 \lra \O(-4) \oplus 2\O(-1) \stackrel{\psi}{\lra} 3\O \lra \F' \lra 0
\]
in which $\psi_{12}= \f_{23}$. According to \ref{3.1.4}, $\F'$ gives a point in
$\M(6,0)$ and the only subsheaf of $\F'$ of slope zero, if there is one, must be of the form $\O_L(-1)$.
It follows that $\E$ must have Hilbert polynomial $\PP_{\E}(m) = 2m+1$, $m+2$ or $m+1$.
If $\PP_{\E}(m)=2m+1$, then $\E$ is the structure sheaf of some conic curve $C \subset \P^2$.
We obtain a commutative diagram with exact rows and injective vertical maps
\[
\xymatrix
{
0 \ar[r] & \O(-2) \ar[r] \ar[d]^-{\b} & \O \ar[r] \ar[d]^-{\a} & \O_C \ar[r] \ar[d] & 0 \\
0 \ar[r] & \O(-3) \oplus \O(-2) \oplus 2\O(-1) \ar[r] & \O(-1) \oplus 3\O \ar[r] & \F \ar[r] & 0
}.
\]
Taking into account the possible canonical forms for $\b$, we see that $\f$ is represented by
a matrix having one of the following forms:
\[
\left[
\ba{cccc}
\star & 0 & 0 & 0 \\
\star & 0 & \star & \star \\
\star & 0 & \star & \star \\
\star & \star & \star & \star
\ea
\right], \qquad \left[
\ba{cccc}
\star & \star & 0 & 0 \\
\star & \star & \star & 0 \\
\star & \star & \star & 0 \\
\star & \star & \star & \star
\ea
\right], \qquad \left[
\ba{cccc}
\star & \star & 0 & 0 \\
\star & \star & 0 & 0 \\
\star & \star & \star & \star \\
\star & \star & \star & \star
\ea
\right].
\]
In each of these situations the hypothesis on $\f$ gets contradicted.
If $\PP_{\E}(m)=m+1$, then $\E$ is the structure sheaf of some line $L \subset \P^2$
and we obtain a contradiction as above.
The case in which $\PP_{\E}(m)= m+2$ is not feasible because in this case $\E \isom \O_L(1)$,
yet $\H^0(\E(-1))$ must vanish because the corresponding group for $\F$ vanishes.
\end{proof}

\begin{prop}
\label{4.1.3}
The sheaves $\F$ giving points in $\M(6,2)$ and satisfying the conditions
$\h^0(\F(-1))=1$ and $\h^1(\F)=1$ are precisely the sheaves having a resolution of the form
\[
0 \lra \O(-3) \oplus 2\O(-2) \stackrel{\f}{\lra} 2\O(-1) \oplus \O(1) \lra \F \lra 0,
\]
where $\f$ satisfies the conditions of claim \ref{3.1.2}.
\end{prop}

\begin{proof}
Let $\F$ give a point in $\M(6,2)$ and satisfy the above cohomological conditions.
Denote $m=\h^0(\F \tensor \Om^1(1))$.
The Beilinson diagram (2.1.4) for the dual sheaf $\G= \F^\D(1)$ giving
a point in $\M(6,4)$ reads
\[
\xymatrix
{
3\O(-2) \ar[r]^-{\f_1} & m\O(-1) \ar[r]^-{\f_2} & \O \\
\O(-2) \ar[r]^-{\f_3} & (m+2)\O(-1) \ar[r]^-{\f_4} & 5\O
}.
\]
Arguing as in the proof of \ref{3.1.6}, we can show that $m=3$, that
$\Ker(\f_2)= \Im(\f_1)$ and $\Ker(\f_1) \isom \O(-3)$.
Combining the exact sequences (2.1.5) and (2.1.6) we get the resolution
\[
0 \lra \O(-2) \stackrel{\psi}{\lra} \O(-3) \oplus 5\O(-1) \lra 5\O \lra \G \lra 0.
\]
As in the proof of 2.1.4 \cite{mult_five}, we have $\Coker(\psi) \isom \O(-3) \oplus 2\O(-1) \oplus \Om^1(1)$
and the cokernel of the induced morphism $\Om^1(1) \to 5\O$ is isomorphic to  $2\O \oplus \O(1)$.
We finally arrive at the resolution dual to the resolution in the proposition:
\[
0 \lra \O(-3) \oplus 2\O(-1) \lra 2\O \oplus \O(1) \lra \G \lra 0.
\]
The converse is the object of claim \ref{3.1.2}.
\end{proof}

\begin{prop}
\label{4.1.4}
The sheaves $\F$ giving points in $\M(6,2)$ and satisfying the conditions
$\h^0(\F(-1))=1$, $\h^1(\F)=2$, $\h^1(\F(1))=0$
are precisely the sheaves having a resolution of the form
\[
0 \lra 2\O(-3) \oplus \O(-1) \stackrel{\f}{\lra} \O(-2) \oplus \O \oplus \O(1) \lra \F \lra 0,
\]
where $\f_{11}$ has linearly independent entries, $\f_{22} \neq 0$ and does not divide $\f_{32}$.
\end{prop}

\begin{proof}
Let $\F$ give a point in $\M(6,2)$ and satisfy the above cohomological conditions.
Display diagram (2.1.1) for the Beilinson spectral sequence I converging to $\F(1)$ reads
\[
\xymatrix
{
5\O(-1) \ar[r]^-{\f_1} & 2\Om^1(1) & 0 \\
\O(-1) \ar[r]^-{\f_3} & 4\Om^1(1) \ar[r]^-{\f_4} & 8\O
}.
\]
Arguing as in the proof of \ref{3.1.5} we see that $\Ker(\f_1) \isom \O(-3)$ and $\Coker(\f_1) \isom \C_x$.
From (2.1.3) we have an extension
\[
0 \lra \F' \lra \F \lra \C_x \lra 0,
\]
where $\F' = \Coker(\f_5)(-1)$. From (2.1.2) we get the exact sequence
\[
0 \lra \O(-2) \lra \O(-4) \oplus 4\Om^1 \lra 8\O(-1) \lra \F' \lra 0,
\]
hence the resolution
\[
0 \lra \O(-2) \lra \O(-4) \oplus 12\O(-1) \stackrel{\rho}{\lra} 8\O(-1) \oplus 4\O \lra \F' \lra 0.
\]
If $\rank(\rho_{12}) \le 7$, then $\F'$ would have a subsheaf of slope $4/3$
that would destabilise $\F$.
Thus $\rank(\rho_{12}) =8$ and we have the resolution
\[
0 \lra \O(-2) \stackrel{\psi}{\lra} \O(-4) \oplus 4\O(-1) \lra 4\O \lra \F' \lra 0.
\]
Arguing as at 2.1.4 \cite{mult_five}, we can show that $\Coker(\psi_{21}) \isom \O(-1) \oplus \Om^1(1)$
and that the cokernel of the induced morphism $\Om^1(1) \to 4\O$ is isomorphic to $\O \oplus \O(1)$.
We obtain the resolution
\[
0 \lra \O(-4) \oplus \O(-1) \lra \O \oplus \O(1) \lra \F' \lra 0.
\]
Combining this with the standard resolution of $\C_x$ tensored with $\O(-2)$ we obtain the exact sequence
\[
0 \lra \O(-4) \lra \O(-4) \oplus 2\O(-3) \oplus \O(-1) \lra \O(-2) \oplus \O \oplus \O(1) \lra \F \lra 0.
\]
The morphism $\O(-4) \to \O(-4)$ is non-zero because $\h^1(\F(1))=0$.
Canceling $\O(-4)$ we obtain a resolution as in the proposition.

Conversely, assume that $\F$ has a resolution as in the proposition.
Then $\F$ is an extension of $\C_x$ by $\F'$, where, in view of \ref{3.1.8},
$\F'$ gives a point in $\M(6,1)$.
It follows that any possibly destabilising subsheaf of $\F$ must be the structure sheaf
of a line or of a conic curve. Each of these situations can be easily ruled out using
diagrams similar to diagram (8) in subsection 3.1.
\end{proof}

\noi
In the remaining part of this subsection we shall prove that there are no
sheaves $\F$ giving points in $\M(6,2)$ beside the sheaves we have discussed in this subsection
and the sheaves at \ref{3.1.7}(ii).
In view of loc.cit., we may restrict our attention to the case when $\H^1(\F(1))=0$.
Assume that $\h^0(\F(-1)) \le 1$. According to \ref{2.2.2}(iv), (v), and \ref{4.1.1}
the pair $(\h^0(\F(-1)), \h^1(\F))$ may be one of the following: $(0, 0)$, $(0, 1)$, $(1,1)$, $(1, 2)$.
Each of these situations has already been examined.
The following concludes the classification of sheaves in $\M(6,2)$:

\begin{prop}
\label{4.1.5}
Let $\F$ be a sheaf giving a point in $\M(6,2)$ and satisfying the condition $\h^1(\F(1))=0$.
Then $\h^0(\F(-1))=0$ or $1$.
\end{prop}

\begin{proof}
Let $\F$ give a point in $\M(6,2)$ and satisfy the condition $\h^0(\F(-1))\ge 2$.
As at 2.1.3 \cite{drezet-maican}, there is an injective morphism $\O_C \to \F(-1)$
for a curve $C \subset \P^2$. This curve has degree $5$ or $6$, otherwise $\O_C$
would destabilise $\F(-1)$.
Assume that $\deg(C)=5$. The quotient sheaf $\CC = \F/\O_C(1)$ has Hilbert polynomial
$\PP(m)=m+2$ and zero-dimensional torsion $\TT$ of length at most $1$.
Indeed, the pull-back in $\F$ of $\TT$ would be a destabilising subsheaf if $\length(\TT) \ge 2$.
If $\TT=0$, then $\CC \isom \O_L(1)$, forcing $\h^0(\F(-1))=2$.
The morphism $\O(1) \to \O_L(1)$ lifts to a morphism $\O(1) \to \F$,
which leads us to the resolution
\[
0 \lra \O(-4) \oplus \O \lra 2\O(1) \lra \F \lra 0.
\]
Thus $\h^1(\F(1))=1$. Assume now that $\length(\TT)=1$.
Let $\F' \subset \F$ be the pull-back of $\TT$.
According to 3.1.5 \cite{mult_five}, we have $\h^0(\F'(-1)) =1$.
Since $\F/\F' \isom \O_L$, we get $\h^0(\F(-1))=1$, contradicting our choice of $\F$.

Assume now that $C$ is a sextic curve.
The quotient sheaf $\CC = \F/\O_C(1)$ is zero-dimensional of length $5$.
Let $\CC' \subset \CC$ be a subsheaf of length $4$ and let $\F'$ be its preimage in $\F$.
We claim that $\F'$ gives a point in $\M(6,1)$. If this were not the case, then $\F'$ would have
a destabilising subsheaf $\F''$, which may be assumed to be semi-stable.
We may assume, without loss of generality, that $\F$ is stable.
Thus we have the inequalities $1/6 < \pp(\F'') < 1/3$.
This leaves only two possibilities: that $\F''$ give a point in $\M(5,1)$ or in $\M(4,1)$.
In the first case $\F/\F''$ is isomorphic to the structure sheaf of a line,
hence $\h^0(\F(-1)) = \h^0(\F''(-1)) = 0$ or $1$, cf. \cite{mult_five}.
This contradicts our choice of $\F$.
In the second case $\F/\F''$ is easily seen to be semi-stable,
hence it is isomorphic to the structure sheaf of a conic curve.
We get $\h^0(\F(-1)) = \h^0(\F''(-1))=0$, cf. \cite{drezet-maican},
contradicting our choice of $\F$. This proves the claim, i.e. that $\F'$ is semi-stable.
We have $\h^0(\F'(-1)) \ge 1$ so, according to the results in subsection 3.1,
there are two possible resolutions for $\F'$:
\[
0 \lra 2\O(-3) \oplus \O(-2) \lra \O(-2) \oplus \O(-1) \oplus \O(1) \lra \F' \lra 0
\]
or
\[
0 \lra \O(-4) \oplus \O(-1) \lra \O \oplus \O(1) \lra \F' \lra 0.
\]
Combining the first resolution with the standard resolution of $\C_x = \CC/\CC'$ tensored
with $\O(1)$ we obtain the exact sequence
\[
0 \lra \O(-1) \lra 2\O(-3) \oplus \O(-2) \oplus 2\O \lra \O(-2) \oplus \O(-1) \oplus 2\O(1) \lra \F \lra 0.
\]
From this it easily follows that $\C_x$ is a direct summand of $\F$, which violates semi-stability.
Assume, finally, that $\F'$ has the second resolution.
We can apply the horseshoe lemma as above, leading to the resolution
\[
0 \lra \O(-1) \lra \O(-4) \oplus \O(-1) \oplus 2\O \lra \O \oplus 2\O(1) \lra \F \lra 0.
\]
We see from this that $\h^1(\F(1))=1$.
\end{proof}

\subsection{The strata as quotients} 

In the previous subsection we classified all sheaves giving points in $\M(6,2)$,
namely we showed that this moduli space can be decomposed into seven subsets $X_0, \ldots, X_6$, cf. Table 2.
For $1 \le i \le 6$, the sheaves giving points in $X_i$ are stable.
We will employ the notations $\W_i$, $W_i$, $G_i$, $0 \le i \le 6$, analogous to the notations
from subsection 3.2.
For $1 \le i \le 6$, the fibres of the canonical maps $\rho_i \colon W_i \to X_i$
are precisely the $G_i$-orbits. It follows, as at 3.2, that these are geometric quotient maps.
The semi-stable but not stable points of $\M(6,2)$ are of the form $[\F_1 \oplus \F_2]$,
where $\F_1$, $\F_2$ give points in $\M(3,1)$, and they are all contained in $X_0$.
Thus $X_0$ cannot be a geometric quotient. Instead, it is a good quotient:

\begin{prop}
\label{4.2.1}
There is a good quotient $W_0 /\!/ G_0$, which is isomorphic to $X_0$.
\end{prop}

\begin{proof}
Let $\W_0^\ss(\L) \subset \W_0$ denote the set of morphisms that are
semi-stable with respect to a polarisation $\L= (\l_1, \m_1, \m_2)$
satisfying the relation $1/8 < \m_2 < 3/16$ (notations as at \cite{drezet-trautmann}).
According to \cite{drezet-2000}, theorem 6.4, $\W_0^\ss(\L)/\!/ G_0$ exists and is a projective variety.
According to \cite{maican}, 4.3, $W_0$ is the subset of injective morphisms inside $\W_0^\ss(\L)$.
Thus $W_0/\!/ G_0$ exists and is a proper open subset of $\W_0^\ss(\L)/\!/ G_0$.

Arguing as at 4.2.1 \cite{drezet-maican},
we can easily see that two points of $W_0$ are in the same fibre of $\rho_0$
if and only if the relative closures in $W_0$ of their $G_0$-orbits intersect non-trivially.
This allows us to apply the method of 4.2.2 op.cit. in order to show that
$\rho_0$ is a categorical quotient map. We need to recover resolution \ref{4.1.1}(i)
from the Beilinson spectral sequence.
Fix $\F$ in $X_0$. Tableau (2.1.4) for the dual sheaf $\F^\D(1)$ reads
\[
\xymatrix
{
2\O(-2) & 0 & 0 \\
0 & 2\O(-1) \ar[r]^-{\f_4} & 4\O
}.
\]
Combining the exact sequences (2.1.5) and (2.1.6) yields the dual to resolution \ref{4.1.1}(i).
Thus $W_0 \to X_0$ is a categorical quotient map and the isomorphism
$W_0/\!/ G_0 \isom X_0$ follows from the uniqueness of the categorical quotient.
\end{proof}

By analogy with 2.2.2 \cite{mult_five}, the quotient $W_1/G_1$
is isomorphic to an open subset of the projectivisation
of a vector bundle over $\N(3,4,3) \times \P^2$ of rank $21$.
By analogy with 3.2.3 op.cit., the quotient $W_3/G_3$ is isomorphic
to an open subset of the projectivisation of a vector bundle over $\Hilb(2) \times \N(3,2,3)$ of rank $23$,
and $W_5/G_5$ is isomorphic to an open subset of the projectivisation
of a vector bundle over $\P^2 \times \Hilb(2)$ of rank $25$.
Recall the smooth projective variety $U/G$ constructed at \ref{3.2.1}.
By analogy with 9.3 \cite{drezet-trautmann}, $W_4/G_4$ is isomorphic to an open subset
of the projectivisation of a vector bundle over $U/G$ of rank $23$.
The smallest stratum $X_6$ is isomorphic to $\Hilb(6,1)$,
i.e. to the universal sextic curve in $\P^2 \times \P(\SS^6 V^*)$.

\subsection{Generic sheaves}  

Let $C \subset \P^2$ denote an arbitrary smooth sextic curve and let $P_i$ denote distinct points on $C$.
According to \cite{modules-alternatives}, propositions 4.5 and 4.6, the cokernels of morphisms
$3\O(-4) \to 4\O(-3)$ whose maximal minors have no common factor are precisely the
ideal sheaves $\I_Z \subset \O_{\P^2}$ of zero-dimensional schemes $Z$ of length $6$ that are not
contained in a conic curve. It follows that the generic sheaves in $X_1$ have the form
$\O_C(1)(P_1 + \cdots + P_6 - P_7)$, where $P_1, \ldots, P_6$ are not contained in a conic curve.
Also from loc.cit. we deduce that the generic sheaves in $X_3$ have the form
$\O_C(2)(-P_1 - P_2 - P_3 + P_4 + P_5)$, where $P_1, P_2, P_3$ are non-colinear.
From \ref{3.3.2} we deduce, by duality, that the generic sheaves in $X_4$ are of the form
$\O_C(1)(P_1 + \cdots + P_5)$, where no three points among $P_1, \ldots, P_5$ are colinear.
It is easy to see that the generic sheaves in $X_5$ are of the form $\O_C(2)(P_1 - P_2 - P_3)$.
According to claim \ref{4.3.1} below, the generic sheaves in $X_2$ have the form
$\O_C(3)(-P_1 - \cdots - P_7)$, where $P_1, \ldots, P_7$ do not lie on a conic curve and no four points
among them are colinear.

\begin{claim}
\label{4.3.1}
Let $U \subset \Hom(\O(-2) \oplus \O(-1), 3\O)$ be the set of morphisms whose maximal minors
have no common factor.
The cokernels of the morphisms in $U$ are precisely the sheaves of the form $\I_Z(3)$,
where $\I_Z \subset \O_{\P^2}$ is the ideal sheaf of a zero-dimensional subscheme $Z \subset \P^2$
of length $7$ that is not contained in a conic curve
and no subscheme of length $4$ of which is contained in a line.
\end{claim}

\begin{proof}
Consider $\psi \in U$.
As the maximal minors of $\psi$, denoted $\z_1$, $\z_2$, $\z_3$,
have no common factor, there is an exact sequence of the form
\[
0 \lra \O(-2) \oplus \O(-1) \stackrel{\psi}{\lra} 3\O \stackrel{\z}{\lra} \O(3) \lra \CC \lra 0,
\]
\[
\z = \left[
\ba{ccc}
\z_1 & -\z_2 & \phantom{-} \z_3
\ea
\right].
\]
The Hilbert polynomial of $\CC$ is $7$, hence $\CC$ is the structure sheaf of a zero-dimensional
scheme $Z$ of length $7$ and $\Coker(\psi) \isom \I_Z(3)$.
If $Z$ were contained in an irreducible conic curve $C$ with equation $q=0$,
then $C$ would meet each of the cubic curves with equation $\z_i =0$ in at least seven points,
hence, by B\'ezout's theorem, $q$ would divide $\z_1$, $\z_2$, $\z_3$, contradicting our hypothesis.
Similarly, if four points of $Z$ lay on the line with equation $\ell=0$, then $\ell$ would divide
the maximal minors of $\psi$.

For the converse we use the method of 4.5 \cite{modules-alternatives}.
Let $Z \subset \P^2$ be a subscheme as in the proposition.
The Beilinson spectral sequence with $\EE^1$-term
\[
\EE^1_{ij} = \H^j (\I_Z(2) \tensor \Om^{-i}(-i)) \tensor \O(i)
\]
converges to $\I_Z(2)$. The bottom and the top row of the display diagram for $\EE^1$
vanish, cf. the arguments at \ref{3.3.1}. The middle row yields a monad
\[
0 \lra 4\O(-2) \lra 6\O(-1) \lra \O \lra 0
\]
with cohomology $\I_Z(2)$. From this we get the resolution
\[
0 \lra 4\O(-2) \lra \Om^1 \oplus 3\O(-1) \lra \I_Z(2) \lra 0.
\]
Resolving $\Om^1$ we obtain the exact sequence
\[
0 \lra \O(-3) \oplus 4\O(-2) \stackrel{\rho}{\lra} 3\O(-2) \oplus 3\O(-1) \lra \I_Z(2) \lra 0.
\]
If $\rho_{12} = 0$, then $\I_Z(2)$ maps surjectively onto $\Coker(\rho_{11})$,
which is absurd because $\rank(\Coker(\rho_{11}))=2$ whereas $\rank(\I_Z(2))=1$.
Assume that $\rank(\rho_{12})=1$. We get a resolution
\[
0 \lra \O(-3) \oplus 3\O(-2) \stackrel{\eta}{\lra} 2\O(-2) \oplus 3\O(-1) \lra \I_Z(2) \lra 0
\]
with $\eta_{12}=0$.
Clearly $\eta_{22}$ is injective and $\Coker(\eta_{22})$ maps injectively to $\I_Z(2)$.
This is absurd because $\Coker(\eta_{22})$ is a torsion sheaf whereas $\I_Z(2)$ is torsion-free.
Assume that $\rank(\rho_{12})=2$. We arrive at a resolution
\[
0 \lra \O(-3) \oplus 2\O(-2) \stackrel{\eta}{\lra} \O(-2) \oplus 3\O(-1) \lra \I_Z(2) \lra 0
\]
with $\eta_{12}=0$.
From the exactness of the above sequence we see that the maximal minors of $\eta$
have no common factor. It follows that the maximal minors of $\eta_{22}$,
denoted $\xi_1$, $\xi_2$, $\xi_3$, have no common factor, too.
Thus $\xi_1$, $\xi_2$, $\xi_3$ generate the ideal sheaf of a zero-dimensional scheme of length $3$.
At least four points of $Z$ are not in the support of this scheme, hence they lie on the line
with equation $\eta_{11}=0$. This contradicts our hypothesis.
We conclude that $\rank(\rho_{12})=3$. Canceling $3\O(-2)$ we obtain the resolution
\[
0 \lra \O(-3) \oplus \O(-2) \stackrel{\psi}{\lra} 3\O(-1) \lra \I_Z(2) \lra 0.
\]
Clearly $\psi$ satisfies the requirements of the proposition.
\end{proof}


\section{The moduli space $\M(6,3)$}

\subsection{Classification of sheaves}   

\begin{prop}
\label{5.1.1}
The sheaves $\F$ giving points in $\M(6,3)$ and satisfying the conditions
$\h^0(\F(-1))=0$, $\h^1(\F)=0$, $\h^1(\F(1))=0$
are precisely the sheaves having one of the following resolutions:
\[
\tag{i}
0 \lra 3\O(-2) \lra 3\O \lra \F \lra 0;
\]
\[
\tag{ii}
0 \lra 3\O(-2) \oplus \O(-1) \stackrel{\f}{\lra} \O(-1) \oplus 3\O \lra \F \lra 0,
\]
where $\f_{12}=0$, the entries of $\f_{11}$ span a subspace of $V^*$ of dimension
at least $2$, the same for the entries of $\f_{22}$ and, moreover, $\f$ is not equivalent
to a morphism represented by a matrix of the form
\[
\left[
\ba{cccc}
\star & \star & 0 & 0 \\
\star & \star & 0 & 0 \\
\star & \star & \star & \star \\
\star & \star & \star & \star
\ea
\right];
\]
\[
\tag{iii}
0 \lra 3\O(-2) \oplus 2\O(-1) \stackrel{\f}{\lra} 2\O(-1) \oplus 3\O \lra \F \lra 0,
\]
where $\f_{12}=0$, $\f_{11}$ has linearly independent maximal minors and the same for $\f_{22}$.
\end{prop}

\begin{proof}
Let $\F$ give a point in $\M(6,3)$ and satisfy the above cohomological conditions.
Diagram (2.1.1) for the Beilinson spectral sequence I converging to $\F(1)$ reads
\[
\xymatrix
{
3\O(-1) & 0 & 0 \\
0 & 3\Om^1(1) \ar[r]^-{\f_4} & 9\O
}.
\]
Combining the exact sequences (2.1.2) and (2.1.3) yields the resolution
\[
0 \lra 3\O(-1) \oplus 3\Om^1(1) \lra 9\O \lra \F(1) \lra 0,
\]
hence the resolution
\[
0 \lra 3\O(-1) \oplus 9\O \stackrel{\rho}{\lra} 9\O \oplus 3\O(1) \lra \F(1) \lra 0.
\]
Note that $\rank(\rho_{12}) \ge 7$, otherwise $\F(1)$ would map surjectively
onto the cokernel of a morphism $3\O(-1) \to 3\O$, in violation of semi-stability.
We now get resolutions (i), (ii), (iii), depending on whether $\rank(\rho_{12}) = 9$, $8$, $7$.
\end{proof}

\begin{prop}
\label{5.1.2}
\emph{(i)} 
The sheaves $\F$ giving points in $\M(6,3)$ and satisfying the conditions $\h^0(\F(-1))=0$, $\h^1(\F)=1$
are precisely the sheaves having a resolution of the form
\[
0 \lra \O(-3) \oplus 3\O(-1) \stackrel{\f}{\lra} 4\O \lra \F \lra 0,
\]
where $\f_{12}$ is semi-stable as a Kronecker module.

\medskip

\noi
\emph{(ii)}
The sheaves $\F$ giving points in $\M(6,3)$ and satisfying the dual conditions $\h^0(\F(-1))=1$, $\h^1(\F)=0$
are precisely the sheaves having a resolution of the form
\[
0 \lra 4\O(-2) \stackrel{\f}{\lra} 3\O(-1) \oplus \O(1) \lra \F \lra 0,
\]
where $\f_{11}$ is semi-stable as a Kronecker module.
\end{prop}

\begin{proof}
Part (i) is a particular case of 5.3 \cite{maican}.
Part (ii) is equivalent to (i) by duality.
\end{proof}

\begin{prop}
\label{5.1.3}
The sheaves $\F$ giving points in $\M(6,3)$ and satisfying the conditions
$\h^0(\F(-1))=1$, $\h^1(\F)=1$, $\h^1(\F(1))=0$
are precisely the sheaves having a resolution of the form
\[
\tag{i}
0 \lra \O(-3) \oplus \O(-2) \stackrel{\f}{\lra} \O \oplus \O(1) \lra \F \lra 0,
\]
where $\f_{12} \neq 0$,
or the sheaves having a resolution of the form
\[
\tag{ii}
0 \lra \O(-3) \oplus \O(-2) \oplus \O(-1) \stackrel{\f}{\lra} \O(-1) \oplus \O \oplus \O(1) \lra \F \lra 0,
\]
where $\f_{12}, \f_{23} \neq 0$, $\f_{12}$ does not divide $\f_{11}$, $\f_{23}$ does not divide $\f_{33}$.
\end{prop}

\begin{proof}
Let $\F$ be a sheaf giving a point in $\M(6,3)$ and satisfying the above cohomological conditions.
Diagram (2.1.1) for the Beilinson spectral sequence I converging to $\F(1)$ takes the form
\[
\xymatrix
{
4\O(-1) \ar[r]^-{\f_1} & \Om^1(1) &  0 \\
\O(-1) \ar[r]^-{\f_3} & 4\Om^1(1) \ar[r]^-{\f_4} & 9\O
}.
\]
Arguing as at \ref{3.1.3} we see that $\Coker(\f_1)=0$ and $\Ker(\f_1) \isom \O(-2) \oplus \O(-1)$.
Performing the same steps as at loc.cit. we arrive at the resolution
\[
0 \lra \O(-1) \lra \O(-2) \oplus \O(-1) \oplus 12\O \stackrel{\rho}{\lra} 9\O \oplus 4\O(1) \lra \F(1) \lra 0.
\]
Notice that $\rank(\rho_{13}) \ge 8$, otherwise $\F(1)$ would map surjectively to the cokernel of a morphism
$\O(-2) \oplus \O(-1) \to 2\O$, in violation of semi-stability.
We arrive at a resolution
\[
0 \lra \O(-2) \stackrel{\psi}{\lra} \O(-3) \oplus \O(-2) \oplus 4\O(-1) \lra \O(-1) \oplus 4\O \lra \F \lra 0
\]
in which $\psi_{11}=0$, $\psi_{21}=0$.
Arguing as in the proof of 2.1.4 \cite{mult_five},
we can show that $\Coker(\psi_{31}) \isom \O(-1) \oplus \Om^1(1)$
and that the cokernel of the induced map $\Om^1(1) \to 4\O$ is isomorphic to $\O \oplus \O(1)$.
We get the resolution
\[
0 \lra \O(-3) \oplus \O(-2) \oplus \O(-1) \stackrel{\f}{\lra} \O(-1) \oplus \O \oplus \O(1) \lra \F \lra 0.
\]
Finally, we obtain resolutions (i) or (ii) depending on whether $\f_{13} \neq 0$ or $\f_{13}=0$.

Conversely, assume that $\F$ has resolution (i). According to \ref{2.3.2}, if $\f_{12}$ and $\f_{22}$ have
no common factor, then $\F$ is semi-stable. If $\f_{12}$ divides $\f_{22}$, then $\F$ is stable-equivalent
to $\O_C \oplus \O_Q(1)$, for a quartic curve $Q$ and a conic curve $C$ in $\P^2$.
It remains to examine the case when $\gcd(\f_{12}, \f_{22})$ is a linear form $\ell$.
In this case we have a non-split extension
\[
0 \lra \O_L(-1) \lra \F \lra \E \lra 0,
\]
where $\E$ has a resolution as at 2.1.4 \cite{mult_five}, so it gives a point in $\M(5,3)$.
It is easy to estimate the slope of any subsheaf of $\F$, showing that this sheaf is semi-stable.

Assume now that $\F$ has resolution (ii). From the snake lemma we get an extension
\[
0 \lra \E \lra \F \lra \O_Z \lra 0,
\]
where $Z$ is the common zero-set of $\f_{11}$ and $\f_{12}$, and $\E$ has a resolution as at \ref{3.1.8},
so it gives a point in $\M(6,1)$.
Assume that $\F' \subset \F$ is a destabilising subsheaf. Since $\pp(\F' \cap \E) \le 0$, we see that $\F'$
has multiplicity at most $3$. By duality, any destabilising subsheaf of $\F^\D(1)$ has multiplicity at most $3$,
hence $\F'$ has multiplicity $3$. Without loss of generality we may assume that $\F'$ gives a point in
$\M(3,2)$. We have a diagram
\[
\xymatrix
{
0 \ar[r] & \O(-2) \oplus \O(-1) \ar[r] \ar[d]^-{\b} & 2\O \ar[r] \ar[d]^-{\a} & \F' \ar[r] \ar[d] & 0 \\
0 \ar[r] & \O(-3) \oplus \O(-2) \oplus \O(-1) \ar[r] & \O(-1) \oplus \O \oplus \O(1) \ar[r] & \F \ar[r] & 0
}
\]
in which either $\a$ and $\b$ are both injective or
\[
\a \sim \left[
\ba{cc}
0 & 0 \\
0 & 0 \\
u_1 & u_2
\ea
\right] \qquad \text{and} \qquad \b \sim \left[
\ba{cc}
0 & 0 \\
1 & 0 \\
0 & 0
\ea
\right] \quad \text{or} \quad \b \sim \left[
\ba{cc}
0 & 0 \\
0 & 0 \\
u & 0
\ea
\right].
\]
Thus $\f_{12}= 0$ or $\f_{23}= 0$, which yields a contradiction.
\end{proof}

\begin{prop}
\label{5.1.4}
The sheaves $\F$ giving points in $\M(6,3)$ and satisfying the conditions
$\h^0(\F(-1))=2$, $\h^1(\F)=2$, $\h^1(\F(1))=0$
are precisely the sheaves having a resolution of the form
\[
0 \lra 2\O(-3) \oplus \O \stackrel{\f}{\lra} \O(-2) \oplus 2\O(1) \lra \F \lra 0,
\]
where $\f_{11}$ has linearly independent entries and the same for $\f_{22}$.
\end{prop}

\begin{proof}
Let $\F$ give a point in $\M(6,3)$ and satisfy the above cohomological conditions.
Display diagram (2.1.1) for the Beilinson spectral sequence I converging to $\F(1)$ reads
\[
\xymatrix
{
5\O(-1) \ar[r]^-{\f_1} & 2\Om^1(1) & 0 \\
2\O(-1) \ar[r]^-{\f_3} & 5\Om^1(1) \ar[r]^-{\f_4} & 9\O
}.
\]
Arguing as in the proof of \ref{3.1.5}, we see that $\Ker(\f_1) \isom \O(-3)$ and $\Coker(\f_1) \isom \C_x$.
From (2.1.3) we have an extension
\[
0 \lra \F' \lra \F \lra \C_x \lra 0,
\]
where $\F' = \Coker(\f_5)(-1)$. From (2.1.2) we get the exact sequence
\[
0 \lra 2\O(-2) \lra \O(-4) \oplus 5\Om^1 \lra 9\O(-1) \lra \F' \lra 0,
\]
hence the resolution
\[
0 \lra 2\O(-2) \lra \O(-4) \oplus 15\O(-1) \stackrel{\rho}{\lra} 9\O(-1) \oplus 5\O \lra \F' \lra 0.
\]
If $\rank(\rho_{12}) \le 8$, then $\F'$ would have a subsheaf of slope $5/3$
that would destabilise $\F$.
Thus $\rank(\rho_{12}) = 9$ and we have the resolution
\[
0 \lra 2\O(-2) \stackrel{\psi}{\lra} \O(-4) \oplus 6\O(-1) \lra 5\O \lra \F' \lra 0.
\]
Arguing as at 3.2.5 \cite{mult_five}, we can show that $\Coker(\psi_{21}) \isom 2\Om^1(1)$.
The exact sequence
\[
0 \lra \O(-4) \oplus 2\Om^1(1) \lra 5\O \lra \F' \lra 0
\]
yields the resolution
\[
0 \lra \O(-4) \oplus 6\O \stackrel{\sigma}{\lra} 5\O \oplus 2\O(1) \lra \F' \lra 0.
\]
If $\rank(\sigma_{12}) \le 4$, then $\F'$ would have a subsheaf of slope $2$ that would destabilise $\F$.
Thus $\rank(\sigma_{12}) = 5$ and we have a resolution
\[
0 \lra \O(-4) \oplus \O \lra 2\O(1) \lra \F' \lra 0.
\]
Combining this with the standard resolution of $\C_x$ tensored with $\O(-2)$ we obtain the exact sequence
\[
0 \lra \O(-4) \lra \O(-4) \oplus 2\O(-3) \oplus \O \lra \O(-2) \oplus 2\O(1) \lra \F \lra 0.
\]
The morphism $\O(-4) \to \O(-4)$ is non-zero because $\h^1(\F(1))=0$.
Canceling $\O(-4)$ we obtain a resolution as in the proposition.

Conversely, assume that $\F$ has a resolution as in the proposition.
Then $\F$ is an extension of $\C_x$ by $\F'$, where, in view of \ref{3.1.7}(ii),
$\F'$ gives a stable point in $\M(6,2)$.
It follows that any possibly destabilising subsheaf of $\F$ must be the structure sheaf of a line.
This situation, however, can be easily ruled out using
a diagram analogous to diagram (8) in subsection 3.1.
\end{proof}

\begin{prop}
\label{5.1.5}
The sheaves $\F$ giving points in $\M(6,3)$ and satisfying the condition $\h^1(\F(1)) > 0$
are precisely the sheaves of the form $\O_C(2)$, where $C \subset \P^2$ is a sextic curve.
\end{prop}

\begin{proof}
The argument is entirely analogous to the argument at 4.1.1 \cite{mult_five}.
\end{proof}

\begin{prop}
\label{5.1.6}
Let $\F$ give a point in $\M(6,3)$ and satisfy the condition $\h^0(\F(-1)) \ge 3$ or the condition
$\h^1(\F) \ge 3$. Then $\F \isom \O_C(2)$ for some sextic curve $C \subset \P^2$.
\end{prop}

\begin{proof}
By Serre duality $\h^1(\F) = \h^0(\F^\D)$,
so it is enough to examine only the case when $\h^0(\F(-1)) \ge 3$.
It is easy to see that $\F$ is stable
(cf. the description in subsection 5.2 of properly semi-stable sheaves).
Arguing as at 2.1.3 \cite{drezet-maican}, we see that there is an injective morphism
$\O_C \to \F(-1)$ for some curve $C \subset \P^2$ of degree at most $6$.
Since $\pp(\O_C) < -1/2$, $C$ has degree $5$ or $6$.
Assume first that $\deg(C)=6$.
The quotient sheaf $\CC = \F/\O_C(1)$ has length $6$ and dimension zero.
Let $\CC' \subset \CC$ be a subsheaf of length $5$ and let $\F'$ be its preimage in $\CC$.
We have an exact sequence
\[
0 \lra \F' \lra \F \lra \C_x \lra 0.
\]
We claim that $\F'$ is semi-stable. If this were not the case, then $\F'$ would have a destabilising
subsheaf $\F''$, which may be assumed to be stable.
In fact, $\F''$ must give a point in $\M(5,2)$ because $1/3 < \pp(\F'') < 1/2$.
According to \cite{mult_five}, section 2, we have the inequality $\h^0(\F''(-1)) \le 1$.
The quotient sheaf $\F/\F''$ has Hilbert polynomial $\PP(m)=m+1$ and no zero-dimensional torsion,
so $\F/\F'' \isom \O_L$.
Thus
\[
\h^0(\F(-1)) \le \h^0(\F''(-1)) + \h^0(\O_L(-1)) \le 1,
\]
contradicting our hypothesis.
This proves that $\F'$ gives a point in $\M(6,2)$.
We have the relation $\h^0(\F'(-1)) \ge 2$ hence,
according to the results in subsection 4.1, there is a resolution
\[
0 \lra \O(-4) \oplus \O \lra 2\O(1) \lra \F' \lra 0.
\]
Combining this with the standard resolution of $\C_x$
tensored with $\O(1)$ we get the exact sequence
\[
\tag{*}
0 \lra \O(-1) \lra \O(-4) \oplus 3\O \lra 3\O(1) \lra \F \lra 0.
\]
From this we obtain the relation $\h^1(\F(1))=1$, hence,
by \ref{5.1.5}, $\F \isom \O_C(2)$.

Assume now that $C$ has degree $5$.
The quotient sheaf $\F/\O_C(1)$ has Hilbert polynomial $\PP(m)=m+3$.
Let $\TT$ denote its zero-dimensional torsion and let $\F'$ be the preimage of $\TT$ in $\F$.
We have $\length(\TT) \le 2$, otherwise $\F'$ would destabilise $\F$.
If $\TT=0$, then $\F/\O_C(1) \isom \O_L(2)$.
We apply the horseshoe lemma to the extension
\[
0 \lra \O_C(1) \lra \F \lra \O_L(2) \lra 0,
\]
to the standard resolution of $\O_C(1)$ and to the resolution
\[
0 \lra \O(-1) \lra 3\O \lra 2\O(1) \lra \O_L(2) \lra 0.
\]
We obtain again resolution (*), hence, as we saw above, $\F \isom \O_C(2)$.
Assume that $\length(\TT) = 1$. According to 3.1.5 \cite{mult_five}, we have $\h^0(\F'(-1))=1$.
Since $\F/\F' \isom \O_L(1)$, we see that $\h^0(\F(-1)) \le 2$, contrary to our hypothesis.
Assume that $\length(\TT) =2$. Since $\F$ is stable, it is easy to see that $\F'$ gives a point in $\M(5,2)$,
so $\h^0(\F'(-1)) \le 1$, forcing $\h^0(\F(-1)) \le 1$, which contradicts our hypothesis.
\end{proof}

\noi
There are no other sheaves giving points in $\M(6,3)$
beside the sheaves we have discussed in this subsection.
To see this we may, by virtue of \ref{5.1.5}, restrict our attention to the case when $\H^1(\F(1))=0$.
According to \ref{5.1.6} and \ref{2.2.2}(vi),
the pair $(\h^0(\F(-1)), \h^1(\F))$ may be one of the following: $(0, 0)$, $(0, 1)$, $(1, 0)$, $(1, 1)$, $(2, 2)$.
Each of these situations has been examined.

\subsection{The strata as quotients}   

In the previous subsection we classified all sheaves giving points in $\M(6,3)$, namely we showed
that this moduli space is the union of nine locally closed subsets, as in Table 3,
which we will call, by an abuse of terminology, strata.
As the notation suggests, the stratum $X_3^\D$ is the image of $X_3$
under the duality automorphism $[\F] \to [\F^\D(1)]$.
The strata $X_i$, $0 \le i \le 7$, $i \neq 3$, are invariant under this automorphism.
We employ the notations $\W_i$, $W_i$, $G_i$, $\rho_i$, $0 \le i \le 7$, analogous to the notations
from subsection 3.2. We denote $W_i^\st = \rho_i^{-1}(X_i^\st)$.
Adopting the notations of \cite{drezet-maican},
let $\E_i$ denote an arbitrary sheaf giving a point in the codimension $i$ stratum of $\M(4,2)$, $i = 0, 1$.
Let $C \subset \P^2$ denote an arbitrary conic curve;
let $Q \subset \P^2$ denote an arbitrary quartic curve.
It is easy to see that all points of the form $[\O_C \oplus \E_i]$
belong to $X_i$ and to no other stratum.
According to \ref{2.3.2}, the set $W_4 \setminus W_4^\st$ consists of those morphisms $\f$
such that $\f_{12}$ divides $\f_{11}$ or $\f_{22}$.
The sheaves $\F = \Coker(\f)$, $\f \in W_4 \setminus W_4^\st$, are precisely the extensions
of $\O_C$ by $\O_Q(1)$ or of $\O_Q(1)$ by $\O_C$ satisfying the conditions
$\h^0(\F(-1))=1$, $\h^1(\F)=1$.
Using the argument found at 3.3.2 \cite{mult_five}, we can show that the extension sheaves
\[
0 \lra \O_Q(1) \lra \F \lra \O_C \lra 0
\]
satisfying the condition $\h^1(\F)=0$ are precisely the sheaves of the form $\Coker(\f)$, $\f \in W_3^\D$
such that the maximal minors of $\f_{11}$ have a common quadratic factor.
By duality, it follows that the extension sheaves
\[
0 \lra \O_C \lra \F \lra \O_Q(1) \lra 0
\]
satisfying the condition $\h^0(\F(-1))=0$ are precisely the sheaves of the form $\Coker(\f)$, $\f \in W_3$
such that the maximal minors of $\f_{12}$ have a common quadratic factor.
This shows that the strata $X_2$, $X_5$, $X_6$, $X_7$ have only stable points
and that the sets $X_3 \setminus X_3^\st$, $X_3^\D \setminus (X_3^\D)^\st$, $X_4 \setminus X_4^\st$
coincide and consist of all points of the form $[\O_C \oplus \O_Q(1)]$.

The fibres of the canonical maps $\rho_i \colon W_i^\st \to X_i^\st$, $0 \le i \le 7$,
are precisely the $G_i$-orbits, hence, by the argument found in subsection 3.2,
these are geometric quotient maps. Thus $X_i \isom W_i/G_i$ for $i \in \{ 2, 5, 6,7 \}$.

Assume that $i \in \{ 0, 1 \}$. Arguing as at 4.2.1 \cite{drezet-maican},
we can easily see that two points of $W_i$ are in the same fibre of $\rho_i$
if and only if the relative closures in $W_i$ of their $G_i$-orbits intersect non-trivially.
This allows us to apply the method of 4.2.2 op.cit. in order to show that
$\rho_i$ is a categorical quotient map.
Note that $W_0$ is a proper invariant open subset of the set of semi-stable Kronecker modules
$3\O(-2) \to 3\O$, so there exists a good quotient $W_0/\!/G_0$ as an open subset of $\N(6,3,3)$.
By the uniqueness of the categorical quotient we have an isomorphism $X_0 \isom W_0/\!/G_0$.
This shows that $\M(6,3)$ and $\N(6,3,3)$ are birational.
Let $W_{10} \subset W_1$ be the open invariant subset given by the condition that
the entries of $\f_{11}$ span $V^*$ and the same for the entries of $\f_{22}$.
Its image $X_{10}$ is open in $X_1$. Since $W_{10} \subset W_1^\st$,
the map $W_{10} \to X_{10}$ is a geometric quotient map.
By analogy with 2.2.2 \cite{mult_five}, the quotient $W_{10}/G_1$ is isomorphic
to an open subset of the projectivisation of a vector bundle of rank $37$
over $\N(3,3,1) \times \N(3,1,3)$. The base is isomorphic to a point,
so $X_{10}$ is an open subset of $\P^{36}$.

By analogy with loc.cit., the quotient $W_2/G_2$ is isomorphic to an open subset
of the projectivisation of a vector bundle of rank $22$ over $\N(3,3,2) \times \N(3,2,3)$.
Likewise, $W_6/G_6$ is isomorphic to an open subset of the projectivisation
of a vector bundle of rank $26$ over $\P^2 \times \P^2$.
By analogy with 3.2.3 op.cit., $W_5/G_5$ is isomorphic to an open subset
of the projectivisation of a vector bundle of rank $24$ over $\Hilb(2) \times \Hilb(2)$.
The stratum $X_7$ is isomorphic to $\P(\SS^6 V^*)$.

By analogy with 9.3 \cite{drezet-trautmann}, there exists a geometric quotient $W_3/G_3$,
which is an open subset of the projectivisation of a vector bundle of rank $22$ over $\N(3,3,4)$.
The induced map $W_3/G_3 \to X_3$ is an isomorphism over the set of stable points
in $X_3$, as we saw above.
We will show that the fibre of this map over any properly semi-stable point $[\O_C \oplus \O_Q(1)]$
is isomorphic to $\SS^3 V^*$. Choose an equation for $C$ of the form
$X \ell_1 + Y \ell_2 + Z \ell_3 =0$ and an equation for $Q$ of the form $X f_3 - Y f_2 + Z f_1 =0$.
For any $f \in \SS^3 V^*$ consider the morphism
\[
\f_f = \left[
\ba{cccc}
f_1 & -Y & \phantom{-} X & 0 \\
f_2 & -Z & \phantom{-} 0 & X \\
f_3 & \phantom{-} 0 & -Z & Y \\
f & \phantom{-} \ell_1 & \phantom{-} \ell_2 & \ell_3
\ea
\right].
\]
Any $\f \in W_3$ such that $\rho_3(\f) = [\O_C \oplus \O_Q(1)]$ is in the orbit of some $\f_f$
and $\f_f \sim \f_g$ if and only if $f = g$.

The linear algebraic group $G= \Aut(\O \oplus \O(1))$ acts on the vector space
$\U = \Hom(\O(-2), \O \oplus \O(1))$ by left-multiplication.
Consider the open $G$-invariant subset $U \subset \U$ of morphisms $\psi$
for which $\psi_{11}$ is non-zero and does not divide $\psi_{21}$.
Consider the fibre bundle with base $\P(\SS^2 V^*)$ and fibre $\P(\SS^3 V^*/V^* q)$
at any point of the base represented by $q \in \SS^2 V^*$.
Clearly this fibre bundle is the geometric quotient of $U$ modulo $G$.
Consider the open $G_4$-invariant subset $W_4' \subset \W_4$ of morphisms $\f$
whose restriction to $\O(-2)$ lies in $U$.
Clearly $W_4'$ is the trivial vector bundle over $U$ with fibre $\Hom(\O(-3), \O \oplus \O(1))$.
Consider the sub-bundle $\Sigma \subset W_4'$ given by the condition
$(\f_{11}, \f_{21}) = (\f_{12} u, \f_{22} u)$, for some $u \in \Hom(\O(-3), \O(-2))$.
As at 2.2.5 \cite{mult_five}, the quotient bundle $W_4'/\Sigma$ is $G$-linearised, 
hence it descends to a vector bundle $E$ over $U/G$ of rank $22$.
Its projectivisation $\P(E)$ is the geometric quotient of $W_4' \setminus \Sigma$
modulo $G_4$.
Notice that $W_4^\st$ is a proper open $G_4$-invariant subset of $W_4' \setminus \Sigma$.
Thus $X_4^\st = W_4^\st/G_4$ is isomorphic to a proper open subset of $\P(E)$.

\subsection{Generic sheaves} 

Let $C$ denote an arbitrary smooth sextic curve in $\P^2$ and let $P_i$ denote distinct points on $C$.
By analogy with the case of the stratum $X_3 \subset \M(6,1)$,
we see that the generic sheaves in $X_2$ have the form $\O_C(2)(-P_1 - P_2 - P_3 + P_4 + P_5 + P_6)$,
where $P_1, P_2, P_3$ are non-colinear and the same for $P_4, P_5, P_6$.
By analogy with the stratum $X_1 \subset \M(6,2)$, we see that the generic sheaves in $X_3$
have the form $\O_C(3)(-P_1 - \cdots - P_6)$,
where $P_1, \ldots, P_6$ are not contained in a conic curve.
The generic sheaves in $X_4$ have the form $\O_C(3)(-P_1 - \cdots - P_6)$,
where $P_1, \ldots, P_6$ lie on a conic curve and no four points among them are colinear.
The generic sheaves in $X_5$ have the form $\O_C(2)(P_1 + P_2 - P_3 - P_4)$.
The generic sheaves in $X_6$ have the form $\O_C(2)(P_1 - P_2)$.


\section{The moduli space $\M(6,0)$}

\subsection{Classification of sheaves}   

\begin{prop}
\label{6.1.1}
Let $r$ be a positive integer.
The sheaves $\F$ giving points in $\M(r,0)$ and satisfying the condition $\h^1(\F)=0$
are precisely the sheaves having a resolution of the form
\[
0 \lra r\O(-2) \stackrel{\f}{\lra} r\O(-1) \lra \F \lra 0.
\]
Moreover, $\F$ is properly semi-stable if and only if $\f$ is properly semi-stable,
viewed as a Kronecker module.
\end{prop}

\begin{proof}
This is a generalisation of 4.1.2 \cite{mult_five}.
Assume that $\F$ gives a point in $\M(r,0)$ and $\h^1(\F)=0$.
Diagram (2.1.4) for the Beilinson spectral sequence II converging to $\F$ reads
\[
\xymatrix
{
r \O(-2) \ar[r]^-{\f_1} & r\O(-1) & 0 \\
0 & 0 & 0
}.
\]
The exact sequences (2.1.5) and (2.1.6) show that $\F \isom \Coker(\f_1)$.
The rest of the proof is exactly as at 4.1.2 \cite{mult_five}.
\end{proof}

\begin{prop}
\label{6.1.2}
Consider an integer $r \ge 3$.
The sheaves $\F$ giving points in $\M(r,0)$ and satisfying the conditions
$\h^0(\F(-1))=0$, $\h^1(\F)=1$, $\h^1(\F(1))=0$
are precisely the sheaves having a resolution of the form
\[
0 \lra \O(-3) \oplus (r-3)\O(-2) \stackrel{\f}{\lra} (r-3)\O(-1) \oplus \O \lra \F \lra 0,
\]
where $\f_{12}$ is semi-stable as a Kronecker module.
\end{prop}

\begin{proof}
This is a generalisation of 4.1.3 \cite{mult_five}.
Assume that $\F$ gives a point in $\M(r,0)$ and satisfies the above cohomological conditions.
Diagram (2.1.1) for the Beilinson spectral sequence I converging to $\F(1)$ reads
\[
\xymatrix
{
r\O(-1) \ar[r]^-{\f_1} & \Om^1(1) & 0 \\
0 & \Om^1(1) \ar[r]^-{\f_4} & r\O
}.
\]
Arguing as at \ref{3.1.3}, we can show that $\Ker(\f_1) \isom \O(-2) \oplus (r-3)\O(-1)$
and $\Coker(\f_1) =0$. By duality, $\Coker(\f_4) \isom (r-3)\O \oplus \O(1)$.
The exact sequence (2.1.3) yields the resolution
\[
0 \lra \O(-2) \oplus (r-3)\O(-1) \lra (r-3)\O \oplus \O(1) \lra \F(1) \lra 0.
\]
The converse is exactly as at 4.1.3 \cite{mult_five}.
\end{proof}

\begin{prop}
\label{6.1.3}
The sheaves $\F$ giving points in $\M(6,0)$ and satisfying the conditions $\h^0(\F(-1))=0$,
$\h^1(\F)=2$ are precisely the sheaves having one of the following resolutions:
\[
\tag{i}
0 \lra 2\O(-3) \lra 2\O \lra \F \lra 0,
\]
\[
\f = \left[
\ba{cc}
f_{11} & f_{12} \\
f_{21} & f_{22}
\ea
\right],
\]
\[
\tag{ii}
0 \lra 2\O(-3) \oplus \O(-2) \stackrel{\f}{\lra} \O(-2) \oplus 2\O \lra \F \lra 0,
\]
\[
\f = \left[
\ba{ccc}
\ell_1 & \ell_2 & 0 \\
f_{11} & f_{12} & q_1 \\
f_{21} & f_{22} & q_2
\ea
\right],
\]
\[
\tag{iii}
0 \lra 2\O(-3) \oplus \O(-1) \stackrel{\f}{\lra} \O(-1) \oplus 2\O \lra \F \lra 0,
\]
\[
\f = \left[
\ba{ccc}
q_1 & q_2 & 0 \\
f_{11} & f_{12} & \ell_1 \\
f_{21} & f_{22} & \ell_2
\ea
\right],
\]
\[
\tag{iv}
0 \lra 2\O(-3) \oplus \O(-2) \oplus \O(-1) \stackrel{\f}{\lra} \O(-2) \oplus \O(-1) \oplus 2\O \lra \F \lra 0,
\]
\[
\f = \left[
\ba{cccc}
\ell_1 & \ell_2 & 0 & 0 \\
p_1 & p_2 & \ell & 0 \\
f_{11} & f_{12} & p_1' & \ell_1' \\
f_{21} & f_{22} & p_2' & \ell_2'
\ea
\right].
\]
Here $q_1$, $q_2$ are linearly independent, $\ell_1$, $\ell_2$ are linearly independent,
$\ell_1'$, $\ell_2'$ are linearly independent and $\ell \neq 0$.
\end{prop}

\begin{proof}
Diagram (2.1.1) for the Beilinson spectral sequence I converging to $\F^\D(2)$ has the form
\[
\xymatrix
{
2\O(-1) & 0 & 0 \\
2\O(-1) \ar[r]^-{\f_3} & 6\Om^1(1) \ar[r]^-{\f_4} & 12 \O
}.
\]
Combining the exact sequences (2.1.2) and (2.1.3), we obtain the resolution
\[
0 \lra 2\O(-1) \lra 2\O(-1) \oplus 6\Om^1(1) \lra 12\O \lra \F^\D(2) \lra 0,
\]
hence the resolution
\[
0 \lra 2\O(-1) \lra 2\O(-1) \oplus 18 \O \stackrel{\rho}{\lra} 12\O \oplus 6\O(1) \lra \F^\D(2) \lra 0.
\]
If $\rank(\rho_{12}) \le 10$, then $\F^\D(2)$ would map surjectively to the cokernel
of a morphism $2\O(-1) \to 2\O$, in violation of semi-stability.
Thus $\rank(\rho_{12}) \ge 11$, which leads us to the resolution
\[
0 \lra 2\O(-1) \stackrel{\psi}{\lra} 2\O(-1) \oplus 7\O \lra \O \oplus 6\O(1) \lra \F^\D(2) \lra 0,
\]
where $\psi_{11}=0$.
Arguing as at 3.1.3 \cite{mult_five}, we see that $\Coker(\psi_{21}) \isom \O \oplus 2\Om^1(2)$.
The resolution
\[
0 \lra 2\O(-3) \oplus \O(-2) \oplus 2\Om^1 \lra \O(-2) \oplus 6\O(-1) \lra \F^\D \lra 0
\]
leads to the exact sequence
\[
0 \lra 2\O(-3) \oplus \O(-2) \oplus 6\O(-1) \stackrel{\eta}{\lra} \O(-2) \oplus 6\O(-1) \oplus 2\O \lra \F^\D \lra 0.
\]
If $\rank(\eta_{23}) \le 4$, then $\F^\D$ would map surjectively to the cokernel of a morphism
$2\O(-3) \oplus \O(-2) \to \O(-2) \oplus 2\O(-1)$, in violation of semi-stability.
Canceling $5\O(-1)$ and dualising yields the resolution
\[
0 \lra 2\O(-3) \oplus \O(-2) \oplus \O(-1) \stackrel{\f}{\lra} \O(-2) \oplus \O(-1) \oplus 2\O \lra \F \lra 0.
\]
From this we get resolutions (i)--(iv) depending on whether $\f_{12} =0$ and $\f_{23}=0$.

\medskip

\noi
Conversely, assume that $\F$ has resolution (i).
If $f_{12}$ and $f_{22}$ have no common factor, then, by virtue of \ref{2.3.2},
$\F$ is semi-stable.
Assume that $\gcd(f_{12}, f_{22})$ is a quadratic polynomial $q$.
We get a non-split extension
\[
0 \lra \O_C(-1) \lra \F \lra \F' \lra 0,
\]
where $C$ is the conic curve given by the equation $q=0$ and $\F'$ has a resolution
\[
0 \lra \O(-3) \oplus \O(-1) \stackrel{\f'}{\lra} 2\O \lra \F' \lra 0
\]
in which the entries of $\f_{12}'$ are linearly independent.
According to \ref{2.3.2}, $\F'$ gives a point in $\M(4,1)$.
It is now easy to see that for any proper subsheaf $\E \subset \F$ we have $\pp(\E) \le 0$.
Assume that $\gcd(f_{12}, f_{22})$ is a linear form $\ell$.
We have an extension
\[
0 \lra \O_L(-2) \lra \F \lra \F' \lra 0,
\]
where $L$ is the line given by the equation $\ell=0$ and $\F'$ has a resolution
\[
0 \lra \O(-3) \oplus \O(-2) \stackrel{\f'}{\lra} 2\O \lra \F' \lra 0
\]
in which the entries of $\f_{12}'$ have no common factor.
According to \ref{2.3.2}, $\F'$ gives a point in $\M(5,1)$.
It is now easy to see that for any proper subsheaf $\E \subset \F$ we have $\pp(\E) \le 0$.

Assume now that $\F$ has resolution (ii) in which $q_1$, $q_2$ have no common factor.
From the snake lemma we get an extension
\[
0 \lra \F' \lra \F \lra \C_x \lra 0,
\]
where $x$ is given by the equations $\ell_1=0$, $\ell_2=0$, and $\F'$ has a resolution
\[
0 \lra \O(-4) \oplus \O(-2) \stackrel{\f'}{\lra} 2\O \lra \F' \lra 0
\]
in which the entries of $\f_{12}'$ have no common factor.
According to \ref{2.3.2}, $\F'$ gives a point in $\M(6,-1)$.
It is now easy to see that for any proper subsheaf $\E \subset \F$ we have $\pp(\E) \le 0$.
If $q_1$ and $q_2$ have a common linear factor, then we have an extension
\[
\tag{*}
0 \lra \O_L(-1) \lra \F \lra \F' \lra 0,
\]
where $\F'$ has a resolution as at 4.1.4 \cite{mult_five}, so is semi-stable.
Thus $\F$ is semi-stable.

Finally, we assume that $\F$ has resolution (iv).
We have an extension
\[
0 \lra \F' \lra \F \lra \C_x \lra 0,
\]
where $x$ is given by the equations $\ell_1=0$, $\ell_2=0$, and $\F'$ has resolution
\[
0 \lra \O(-4) \oplus \O(-2) \oplus \O(-1) \stackrel{\f'}{\lra} \O(-1) \oplus 2\O \lra \F' \lra 0,
\]
\[
\f' = \left[
\ba{ccc}
\star & \ell & 0 \\
\star & p_1' & \ell_1' \\
\star & p_2' & \ell_2'
\ea
\right]. \quad \text{Assume first that} \quad \f' \nsim \left[
\ba{ccc}
\star & \star & 0 \\
\star & 0 & \star \\
\star & 0 & \star
\ea
\right].
\]
Then, according to \ref{3.1.6}(ii), $\F'$ is semi-stable, showing that $\F$ is semi-stable.
If $\f'$ has the special form given above, then we have extension (*), showing that $\F$ is semi-stable.
\end{proof}

\begin{prop}
\label{6.1.4}
\emph{(i)}
The sheaves $\F$ giving points in $\M(6,0)$ and satisfying the conditions
$\h^0(\F(-1)) > 0$, $\h^1(\F(1)) = 0$
are precisely the sheaves having a resolution of the form
\[
0 \lra 3\O(-3) \stackrel{\f}{\lra} 2\O(-2) \oplus \O(1) \lra \F \lra 0,
\]
where $\f_{11}$ has linearly independent maximal minors.

\medskip

\noi
\emph{(ii)}
The sheaves $\F$ giving points in $\M(6,0)$ and satisfying the dual conditions
$\h^0(\F(-1)) = 0$, $\h^1(\F(1)) > 0$
are precisely the sheaves having a resolution of the form
\[
0 \lra \O(-4) \oplus 2\O(-1) \stackrel{\f}{\lra} 3\O \lra \F \lra 0,
\]
where $\f_{12}$ has linearly independent maximal minors.
\end{prop}

\begin{proof}
Part (ii) is equivalent to (i) by duality, so we concentrate on (i).
Assume that $\F$ gives a point in $\M(6,0)$ and satisfies the cohomological conditions from (i).
There is an injective morphism $\O_C \to \F(-1)$ for some curve $C \subset \P^2$.
Note that $\deg(C)= 5$ or $6$, otherwise the semi-stability of $\F(-1)$ would be contradicted.
Assume first that $\deg(C) = 5$.
Let $\TT$ denote the zero-dimensional torsion of $\F/\O_C(1)$.
If $\TT \neq 0$, then the pull-back of $\TT$ in $\F$ would be a destabilising subsheaf.
Thus $\TT=0$, hence $\F/\O_C(1) \isom \O_L(-1)$.
We apply the horseshoe lemma to the extension
\[
0 \lra \O_C(1) \lra \F \lra \O_L(-1) \lra 0,
\]
to the standard resolution of $\O_C(1)$ and to the resolution
\[
0 \lra \O(-4) \lra 3\O(-3) \lra 2\O(-2) \lra \O_L(-1) \lra 0.
\]
We obtain a resolution
\[
0 \lra \O(-4) \lra \O(-4) \oplus 3\O(-3) \lra 2\O(-2) \oplus \O(1) \lra \F \lra 0.
\]
Since $\h^1(\F(1))=0$, the map $\O(-4) \to \O(-4)$ is non-zero.
Canceling $\O(-4)$ we arrive at a morphism as in the proposition.

Assume now that $\deg(C)=6$.
The quotient sheaf $\CC = \F/\O_C(1)$ has dimension zero and length $3$.
Let $\C_x \subset \CC$ be a subsheaf of length $1$ and let $\F' \subset \F$ be its preimage.
We apply the horseshoe lemma to the extension
\[
0 \lra \O_C(1) \lra \F' \lra \C_x \lra 0,
\]
to the standard resolution of $\O_C(1)$ and to the standard resolution of $\C_x$ tensored with $\O(-3)$.
We obtain the resolution
\[
0 \lra \O(-5) \lra \O(-5) \oplus 2\O(-4) \lra \O(-3) \oplus \O(1) \lra \F' \lra 0.
\]
The morphism $\O(-5) \to \O(-5)$ is non-zero otherwise, arguing as at 2.3.2 \cite{mult_five},
we would deduce that $\C_x$ is a direct summand of $\F'$, which is absurd.
We have an extension
\[
0 \lra \F' \lra \F \lra \CC' \lra 0,
\]
where $\CC'$ has length $2$. Let $\C_y \subset \CC'$ be a subsheaf of length $1$
and let $\F'' \subset \F$ be its preimage.
We apply the horseshoe lemma to the extension
\[
0 \lra \F' \lra \F'' \lra \C_y \lra 0,
\]
to the standard resolution of $\C_y$ tensored with $\O(-2)$ and to the resolution
\[
0 \lra 2\O(-4) \lra \O(-3) \oplus \O(1) \lra \F' \lra 0.
\]
We obtain a resolution
\[
0 \lra \O(-4) \lra 2\O(-4) \oplus 2\O(-3) \lra \O(-3) \oplus \O(-2) \oplus \O(1) \lra \F'' \lra 0
\]
in which, by the same argument as above, the morphism $\O(-4) \to 2\O(-4)$ is non-zero.
Canceling $\O(-4)$ we obtain the resolution
\[
0 \lra \O(-4) \oplus 2\O(-3) \lra \O(-3) \oplus \O(-2) \oplus \O(1) \lra \F'' \lra 0.
\]
The morphism $2\O(-3) \to \O(-3)$ is non-zero, otherwise $\F$ would have a destabilising
subsheaf that is the cokernel of a morphism $2\O(-3) \to \O(-2) \oplus \O(1)$.
Denote $\C_z = \CC'/\C_y$.
We apply the horseshoe lemma to the extension
\[
0 \lra \F'' \lra \F \lra \C_z \lra 0,
\]
to the standard resolution of $\C_z$ tensored with $\O(-2)$ and to the resolution
\[
0 \lra \O(-4) \oplus \O(-3) \lra \O(-2) \oplus \O(1) \lra \F'' \lra 0.
\]
We arrive at the resolution
\[
0 \lra \O(-4) \lra \O(-4) \oplus 3\O(-3) \lra 2\O(-2) \oplus \O(1) \lra \F \lra 0.
\]
The morphism $\O(-4) \to \O(-4)$ is non-zero because $\h^1(\F(1))=0$.
Canceling $\O(-4)$ we obtain a resolution as in the proposition.

Conversely, if $\F$ has resolution (i) or (ii), then, by virtue of claim \ref{3.1.4},
$\F$ is semi-stable.
\end{proof}

\begin{prop}
\label{6.1.5}
The sheaves $\F$ giving points in $\M(6,0)$ and satisfying the conditions
$\h^0(\F(-1))>0$, $\h^1(\F(1))>0$ are precisely the sheaves having a resolution
of the form
\[
0 \lra \O(-4) \oplus \O(-2) \stackrel{\f}{\lra} \O(-1) \oplus \O(1) \lra \F \lra 0, \qquad
\text{where $\f_{12} \neq 0$}.
\]
\end{prop}

\begin{proof}
Arguing as in the proof of \ref{6.1.4}, we see that there is an extension
\[
0 \lra \O_C(1) \lra \F \lra \O_L(-1) \lra 0
\]
or there is a resolution
\[
0 \lra \O(-4) \stackrel{\psi}{\lra} \O(-4) \oplus 3\O(-3) \stackrel{\rho}{\lra} 2\O(-2) \oplus \O(1) \lra \F \lra 0.
\]
In the first case we can combine the standard resolutions of $\O_C(1)$ and $\O_L(-1)$
to get a resolution as in the proposition.
Indeed, by hypothesis $\h^0(\F(1)) \ge 7$,
hence $\F(1)$ has a section mapping to a non-zero section of $\O_L$.

In the second case $\psi_{11} = 0$ because $\h^1(\F(1)) > 0$.
We claim that $\Coker(\psi_{21}) \isom \Om^1(-1)$,
i.e. that the entries of $\psi_{21}$ are linearly independent.
Clearly they span a vector space of dimension at least $2$.
If
\[
\psi_{21} \sim \left[
\ba{c}
0 \\ \star \\ \star
\ea
\right], \qquad \text{then} \qquad \rho \sim
\left[
\ba{cccc}
\star & \star & \star & \star \\
\star & \star & 0 & 0 \\
\star & \star & 0 & 0
\ea
\right].
\]
It would follow that $\F$ maps surjectively to the cokernel of an injective morphism
$\O(-4) \oplus \O(-3) \to \O(-2) \oplus \O(1)$.
This would contradict the semi-stability of $\F$. From the resolution
\[
0 \lra \O(-4) \oplus \Om^1(-1) \lra 2\O(-2) \oplus \O(1) \lra \F \lra 0
\]
we obtain the resolution
\[
0 \lra \O(-4) \oplus 3\O(-2) \stackrel{\f}{\lra} 2\O(-2) \oplus \O(-1) \oplus \O(1) \lra \F \lra 0.
\]
If $\rank(\f_{12}) \le 1$, then $\F$ would map surjectively to $\O_C(-2)$ for a conic curve $C \subset \P^2$,
in violation of semi-stability.
Thus $\rank(\f_{12}) =2$ and canceling $2\O(-2)$ we obtain a resolution as in the proposition.
\end{proof}

\noi
In view of \ref{2.2.2}(vii) there are no other sheaves giving points in $\M(6,0)$ beside the sheaves
we have discussed in this subsection.

\subsection{The strata as quotients}  

In the previous subsection we classified all sheaves giving points in $\M(6,0)$,
namely we showed that this moduli space is the union of six locally closed subsets as in Table 4.
As the notation suggests,
$X_3^\D$ is the image of $X_3$ under the duality automorphism $[\F] \to [\F^\D]$.
The strata $X_i$, $i= 0, 1, 2, 4$, are invariant under this automorphism.
We employ the notations $\W_i$, $W_i$, $G_i$, $\rho_i$, $0 \le i \le 4$, analogous to the notations
from subsection 3.2. We denote $W_i^\st = \rho_i^{-1}(X_i^\st)$.

From proposition \ref{6.1.1} we easily deduce that the points in $X_0$ given by properly semi-stable
sheaves are of the form $[\F_1 \oplus \cdots \oplus \F_{\k}]$, where $\F_i$ is stable and has resolution
\[
0 \lra r_i \O(-2) \lra r_i \O(-1) \lra \F_i \lra 0,
\]
$r_1 + \cdots + r_{\k} =6$. In particular, we see that $X_0$ is disjoint from the other strata
and that two points in $W_0$ are in the same fibre of $\rho_0$ if and only if the relative closures in $W_0$
of their orbits meet non-trivially.
This allows us to apply the method of 4.2.2 \cite{drezet-trautmann} in order to show that
$\rho_0$ is a categorical quotient map.
Note that $W_0$ is a proper invariant open subset of the set of semi-stable Kronecker modules
$6\O(-2) \to 6\O(-1)$, so there exists a good quotient $W_0/\!/G_0$ as an open subset of $\N(3,6,6)$.
By the uniqueness of the categorical quotient we have an isomorphism $X_0 \isom W_0/\!/G_0$.
This shows that $\M(6,0)$ and $\N(3,6,6)$ are birational.

According to \ref{3.1.4} and \ref{2.3.2},
the sets $X_3 \setminus X_3^\st$, $X_3^\D \setminus (X_3^\D)^\st$, $X_4 \setminus X_4^\st$
coincide and consist of all points of the form $[\O_L(-1) \oplus \O_Q(1)]$,
where $L \subset \P^2$ is a line and $Q \subset \P^2$ is a quintic curve.
Moreover, from the proofs of \ref{6.1.4} and \ref{6.1.5} it transpires that any sheaf stable-equivalent to
$\O_L(-1) \oplus \O_Q(1)$ is the cokernel of some morphism in
$W_3 \setminus W_3^\st$, $W_3 \setminus (W_3^\D)^\st$ or $W_4 \setminus W_4^\st$.
Thus $X_3$, $X_3^\D$, $X_4$ are disjoint from the strata $X_0$, $X_1$ and $X_2$.
For $i \in \{ 3, 4\}$ the fibres of the map $W_i^\st \to X_i^\st$ are precisely the $G_i$-orbits hence,
as at 3.2, this is a geometric quotient map.
According to 9.3 \cite{drezet-trautmann}, $W_3^\st/G_3$ is an open subset of a fibre bundle
over $\N(3,2,3)$ with fibre $\P^{24}$.
We can be more precise.
According to \ref{3.1.4}, $W_3^\st$ is the subset of $W_3$ of morphisms $\f$
such that the maximal minors of $\f_{12}$ have no common factor, hence,
applying \cite{modules-alternatives}, propositions 4.5 and 4.6,
we can show that the sheaves $\Coker(\f)$, $\f \in W_3^\st$,
are precisely the sheaves of the form $\J_Z(2)$, where $Z \subset \P^2$ is a zero-dimensional
scheme of length $3$ that is not contained in a line, $Z$ is contained in a sextic curve $C$
and $\J_Z \subset \O_C$ is the ideal of $Z$ in $C$.
Thus $W_3^\st/G_3$ is isomorphic to the open subset
of $\Hilb(6,3)$ of pairs $(C,Z)$ such that $Z$ is not contained in a line.
Similarly, the sheaves giving points in $X_4^\st$ are precisely the sheaves of the form
$\J_Z(2)$, where $Z$ is contained in a line $L$ that is not a component of $C$.
Thus $X_4$ is isomorphic to the locally closed subset $\{ (C,Z), \ Z \subset L, \ L \nsubseteq C \}$ of $\Hilb(6,3)$.

By the discussion above, the strata $X_1$ and $X_2$ are disjoint from $X_0$, $X_3$, $X_3^\D$, $X_4$.
We claim that
\[
X_1 \cap X_2 = \{ [\O_{C_1} \oplus \O_{C_2}], \quad \text{$C_1, C_2 \subset \P^2$ cubic curves} \}.
\]
The r.h.s. is clearly contained in $X_1 \cap X_2$.
To prove the reverse inclusion we will make a list of properly semi-stable sheaves $\F$ giving points in $\M(6,0)$.
If $\F$ is stable-equivalent to a direct sum of stable sheaves,
we call the \emph{type} of $\F$ the tuple of multiplicities of these sheaves.
We may assume that the type of $\F$ is a tuple of non-decreasing integers.
A properly semi-stable sheaf in $\M(6,0)$ may have one of the following types:
$(1,5)$, $(2,4)$, $(3,3)$, $(1,1,4)$, $(1,2,3)$, $(2,2,2)$, $(1,1,1,3)$, $(1,1,2,2)$, $(1,1,1,1,2)$, $(1,1,1,1,1,1)$.
The sheaves in $\M(1,0)$ and $\M(2,0)$ are cokernels of morphisms $r\O(-2) \to r\O(-1)$, $r=1,2$.
In view of \ref{6.1.1}, this shows that the last three types belong to $X_0$, and so does type $(2,2,2)$.
Let $L \subset \P^2$ denote an arbitrary line;
let $\A$ denote an arbitrary stable sheaf in $\M(2,0)$;
let $\B$ denote an arbitrary stable sheaf in $\M(3,0)$ that is the cokernel of a morphism $3\O(-2) \to 3\O(-1)$;
let $C \subset \P^2$ denote an arbitrary cubic curve;
let $\CC$ denote an arbitrary sheaf giving a point in the stratum $X_1 \subset \M(4,0)$, cf. 5.2 \cite{drezet-maican};
let $\E_i$ denote an arbitrary stable sheaf giving a point in the stratum $X_i \subset \M(5,0)$, $i = 1,2$,
cf. 4.1 \cite{mult_five};
let $\HH$ denote an arbitrary sheaf giving a point in $\M(2,0)$;
let $\G$ denote an arbitrary sheaf giving a point in $\M(3,0)$ that is the cokernel of a morphism
$3\O(-2) \to 3\O(-1)$.
Eliminating all sheaves giving points in the strata $X_0$, $X_3$, $X_3^\D$, $X_4$ of $\M(6,0)$
leaves us with the following possibilities:
\[
\O_L(-1) \oplus \E_1, \qquad
\O_L(-1) \oplus \E_2, \qquad
\A \oplus \CC, \qquad
\B \oplus \O_C, \qquad
\O_{C_1} \oplus \O_{C_2},
\]
\[
\O_{L_1}(-1) \oplus \O_{L_2}(-1) \oplus \CC, \qquad \qquad
\O_L(-1) \oplus \A \oplus \O_C,
\]
\[
\O_{L_1}(-1) \oplus \O_{L_2}(-1) \oplus \O_{L_3}(-1) \oplus \O_C.
\]
Thus $\F$ is one of the following:
\[
\O_L(-1) \oplus \E_1, \qquad
\O_L(-1) \oplus \E_2, \qquad
\HH \oplus \CC, \qquad
\G \oplus \O_C, \qquad
\O_{C_1} \oplus \O_{C_2}.
\]
Combining the standard resolution for $\O_L(-1)$ with the resolution
\[
0 \lra \O(-3) \oplus 2\O(-2) \lra 2\O(-1) \oplus \O \lra \E_1 \lra 0
\]
we obtain resolution \ref{6.1.2}. Combining the standard resolution of $\O_L(-1)$ with the resolution
\[
0 \lra 2\O(-3) \oplus \O(-1) \lra \O(-2) \oplus 2\O \lra \E_2 \lra 0
\]
we obtain resolution \ref{6.1.3}(iv).
Combining the resolutions
\[
0 \lra 2\O(-2) \lra 2\O(-1) \lra \HH \lra 0,
\]
\[
0 \lra \O(-3) \oplus \O(-2) \lra \O(-1) \oplus \O \lra \CC \lra 0
\]
we obtain resolution \ref{6.1.2}. Combining the standard resolution of $\O_C$ with the resolution
\[
0 \lra 3\O(-2) \lra 3\O(-1) \lra \G \lra 0
\]
we obtain resolution \ref{6.1.2}.
Thus $[\O_L(-1) \oplus \E_1]$, $[\HH \oplus \CC]$, $[\G \oplus \O_C]$
belong to $X_1 \setminus X_2$, $[\O_L(-1) \oplus \E_2]$ belongs to $X_2 \setminus X_1$,
which proves the claim. Denote $X_{10}= X_1 \setminus X_2$, $X_{20}= X_2 \setminus X_1$,
$W_{10}= \rho_1^{-1}(X_{10})$, $W_{20}= \rho_2^{-1}(X_{20})$.
Assume that $i \in \{ 1, 2 \}$. Arguing as at 4.2.1 \cite{drezet-maican},
we can easily see that two points of $W_{i0}$ are in the same fibre of $\rho_i$
if and only if the relative closures in $W_{i0}$ of their $G_i$-orbits intersect non-trivially.
This allows us to apply the method of 4.2.2 op.cit. in order to show that the maps
$W_{i0} \to X_{i0}$ are categorical quotient maps.

\subsection{Generic sheaves}  

Let $C \subset \P^2$ denote an arbitrary smooth sextic curve and let $P_i$ denote distinct points on $C$.
According to \cite{modules-alternatives}, propositions 4.5 and 4.6, the cokernels of morphisms
$5\O(-6) \to 6\O(-5)$ whose maximal minors have no common factor are precisely the ideal
sheaves $\I_Z \subset \P^2$ of zero-dimensional schemes $Z$ of length $15$ that are not contained
in a quartic curve. It follows that the generic sheaves in $X_0$ have the form
$\O_C(4)(-P_1 - \cdots - P_{15})$, where $P_1, \ldots, P_{15}$ are not contained in a quartic curve.
From claim \ref{6.3.1} below, it follows that the generic sheaves in $X_1$ have the form
$\O_C(3)(-P_1 - \cdots - P_9)$, where $P_1, \ldots, P_9$ are contained in a unique cubic curve.
It is obvious that the generic sheaves in $X_2$ have the form
$\O_C(3)(-P_1 - \cdots - P_9)$, where $P_1, \ldots, P_9$
are contained in two cubic curves that have no common component.
We saw in the previous subsection that the generic sheaves in $X_3$ have the form $\O_C(2)(-P_1-P_2-P_3)$,
and the generic sheaves in $X_3^\D$ have the form $\O_C(1)(P_1+P_2+P_3)$, where $P_1, P_2, P_3$
are non-colinear. The generic sheaves in $X_4$ have the form $\O_C(2)(-P_1-P_2-P_3)$,
where $P_1, P_2, P_3$ are colinear.

\begin{claim}
\label{6.3.1}
Let $U \subset \Hom(3\O(-2), 3\O(-1) \oplus \O)$ be the set of morphisms whose maximal minors
have no common factor. The cokernels of morphisms in $U$ are precisely the sheaves of the form
$\I_Z(3)$, where $\I_Z \subset \O_{\P^2}$ is the ideal sheaf of a zero-dimensional scheme $Z$
of length $9$, that is contained in a unique cubic curve.
\end{claim}

\begin{proof}
Take $\psi \in U$ and let $\z_1$, $\z_2$, $\z_3$, $\z_4$ be its maximal minors.
We have an exact sequence
\[
0 \lra 3\O(-2) \stackrel{\psi}{\lra} 3\O(-1) \oplus \O \stackrel{\z}{\lra} \O(3) \lra \O_Z(3) \lra 0,
\]
\[
\z = \left[
\ba{cccc}
\z_1 & -\z_2 & \phantom{-}\z_4 & -\z_4
\ea
\right],
\]
where $Z \subset \P^2$ is the scheme defined by the ideal $(\z_1, \z_2, \z_3, \z_4)$.
Since $\PP_{\O_Z(3)}=9$, we deduce that $Z$ is a zero-dimensional scheme of length $9$
and that $\Coker(\psi) \isom \I_Z(3)$.
We have $\h^0(\I_Z(3))=1$, i.e. $Z$ is contained in a unique cubic curve.

Conversely, assume we are given $Z$ as in the proposition.
We have $\H^0(\I_Z(2))=0$. Indeed, if $Z$ were contained in a conic curve, then we could
produce at least two distinct cubic curves containing $Z$.
Similarly, $\H^0(\I_Z(1))=0$.
By Serre duality $\H^2(\I_Z(i))=0$ for $i = 1,2,3$.
The $\EE^1$-term for the Beilinson spectral sequence I converging to $\I_Z(3)$
has display diagram
\[
\xymatrix
{
0 & 0 & 0 \\
6\O(-1) \ar[r]^-{\f_1} & 3\Om^1(1) & 0 \\
0 & 0 & \O
}.
\]
We have exact sequences
\[
0 \lra \Ker(\f_1) \stackrel{\f_5}{\lra} \O \lra \I_Z(3) \lra \Coker(\f_1) \lra 0,
\]
\[
0 \lra \Ker(\f_1) \lra 3\O(-2) \oplus 6\O(-1) \stackrel{\rho}{\lra} 9\O(-1) \lra \Coker(\f_1) \lra 0.
\]
Assume that $\rank(\Ker(\f_1))=1$. Thus $\Coker(\f_5)$ is a torsion sheaf.
Since $\Coker(\f_5)$ is a subsheaf of the torsion-free sheaf $\I_Z(3)$, we deduce that
$\Coker(\f_5)=0$, i.e. that $\f_5$ is an isomorphism.
This is absurd, $\O$ cannot be a subsheaf of $6\O(-1)$.
This proves that $\Ker(\f_1)=0$, hence $\rank(\Coker(\f_1))=0$,
from which we deduce that $\rank(\rho_{12})=6$.
Combining the exact sequences
\[
0 \lra \O \lra \I_Z(3) \lra \Coker(\f_1) \lra 0
\]
and
\[
0 \lra 3\O(-2) \lra 3\O(-1) \lra \Coker(\f_1) \lra 0
\]
we obtain the resolution
\[
0 \lra 3\O(-2) \stackrel{\psi}{\lra} 3\O(-1) \oplus \O \lra \I_Z(3) \lra 0.
\]
It is clear that the maximal minors of $\psi$ have no common factor.
\end{proof}

\end{document}